\documentclass[a4paper,reqno]{amsart}

\usepackage[foot]{amsaddr}
\usepackage{amsmath,amssymb,amsfonts,amsthm}
\usepackage[colorlinks=true,citecolor=blue]{hyperref}
\usepackage[utf8]{inputenc}
\usepackage{enumerate}
\usepackage{booktabs}
\usepackage{tikz,calc}
\usepackage{cleveref}
\usepackage{siunitx}

\usepackage[
maxbibnames=99,
style=ext-numeric,
backend=biber,
articlein=false,
sorting=nyt,
sortcites,
natbib=false,
doi=false,
url=false,
isbn=false,
eprint=false]{biblatex}

\theoremstyle{plain}
\newtheorem{theorem}{Theorem}
\newtheorem{proposition}{Proposition}

\newtheorem{corollary}{Corollary}
\newtheorem{lemma}{Lemma}[section]

\theoremstyle{definition}
\newtheorem{definition}{Definition}[section]
\newtheorem{example}{Example}[section]

\newtheorem{problem}{Problem}

\newcommand{\nats}{\mathbb{N}}

\newcommand{\abs}[1]{\left\lvert#1\right\rvert}

\let\leq\leqslant
\let\geq\geqslant
\let\phi\varphi

\title{Sizes of flat maximal antichains of subsets}

\author[J.R. Griggs]{Jerrold R. Griggs$^1$}
\thanks{The research of Jerrold Griggs is supported in part by grant \#282896 from the Simons Foundation.}
\author[T. Kalinowski]{Thomas Kalinowski$^2$}
\author[U. Leck]{Uwe Leck$^3$}
\author[I.T. Roberts]{Ian T. Roberts$^4$}
\author[M. Schmitz]{Michael Schmitz$^3$}

\address{$^1$Department of Mathematics, University of South Carolina, Columbia, SC 29212, USA }
\address{$^2$Institut f\"ur Mathematik, Universit\"at Rostock, Germany}
\address{$^3$Institut f\"ur Mathematik, Europa-Universit\"at Flensburg, Germany}
\address{$^4$Faculty of Arts and Society, Charles Darwin University, Darwin, NT, Australia}

\email[J.R.~Griggs]{griggs@math.sc.edu}
\email[T.~Kalinowski]{thomas.kalinowski@uni-rostock.de}
\email[U.~Leck]{uwe.leck@uni-flensburg.de}
\email[I.T.~Roberts]{ian.roberts@cdu.edu.au}
\email[M.~Schmitz]{michael.schmitz@uni-flensburg.de}

\date{\today}

\begin{document}

\begin{abstract}
  This is the second of two papers investigating for which positive integers $m$ there exists a maximal antichain of size $m$ in the Boolean lattice $B_n$ (the power set of $[n]:=\{1,2,\dots,n\}$, ordered by inclusion). In the first part, the sizes of maximal antichains have been characterized. Here we provide an alternative construction with the benefit of showing that almost all sizes of maximal antichains can be obtained using antichains containing only $l$-sets and $(l+1)$-sets for some $l$.
\end{abstract}

\maketitle

\section{Introduction}\label{sec:intro}

In the Boolean lattice $B_n$ of all subsets of $[n]:=\{1,2,\dots,n\}$, ordered by inclusion, an \emph{antichain} is a family of subsets such that no one contains any other one. By a classical theorem of Sperner~\cite{Sperner1928}, the maximum size of an antichain in $B_n$ is $\binom{n}{\lfloor n/2 \rfloor}$, and it is obtained only by the collections $\binom{[n]}{\lfloor n/2\rfloor}$ and $\binom{[n]}{\lceil n/2\rceil}$, where we write $\binom{[n]}{l}$ for the $l$-th \emph{level}\/ of $B_n$, i.e., for the collection of all $l$-element subsets of $[n]$. An antichain is called \emph{maximal} if no $B\subseteq[n]$ with $B\not\in\mathcal A$ can be added to $\mathcal A$ without destroying the antichain property. Let $S(n)=\{\abs{\mathcal A}\,:\,\mathcal A\text{ is a maximal antichain in }B_n\}$. In~\cite{Griggs2023}, we have proved the following characterization for the sets $S(n)$.
\begin{theorem}[\cite{Griggs2023}]\label{thm:MAC_sizes}
  Let $k=\lceil n/2\rceil$, and let $m$ be an integer with $1\leq m\leq\binom{n}{k}$. Then $m\in S(n)$ if and only if $m\leq\binom{n}{k}-k\left\lceil(k+1)/2\right\rceil$ or $m$ has the form
  \[m=\binom{n}{k}-tl+\binom{a}{2}+\binom{b}{2}+c\]
  for some integers $l\in\{k,n-k\}$, $t\in\{0,\dots,k\}$, $a\geq b\geq c\geq 0$, $a+b\leq t$. Moreover, if $m\in S(n)$ satisfies $m\geq\binom{n}{k}-k^2$ then there exists a maximal antichain $\mathcal A$ with $\abs{\mathcal A}=m$ and $\mathcal A\subseteq\binom{[n]}{k}\cup\binom{[n]}{k+1}$ or $\mathcal A\subseteq\binom{[n]}{k-1}\cup\binom{[n]}{k}$.
\end{theorem}
The last statement in Theorem~\ref{thm:MAC_sizes} says that all sizes close to the maximum possible size $\binom{n}{k}$ are obtained by \emph{flat} antichains, which means that there exists an integer $l$ such that $\abs{A}\in\{l,l+1\}$ for every $A\in\mathcal A$ . In the present paper, we add this as a constraint and investigate for which $m\in S(n)$ there exists a \emph{flat} maximal antichain of size $m$. Our main result is the following answer to this question.
\begin{theorem}\label{thm:main_result}
  Let $k=\lceil n/2\rceil$, and let $m$ be an integer with
\[\binom{n}{2}-\left\lfloor\frac{(n+1)^2}{8}\right\rfloor\leq m\leq\binom{n}{k}-k\left\lceil\frac{k+1}{2}\right\rceil.\]
Then there exist an $l\in\{2,3,\dots,\left\lfloor(n-2)/2\right\rfloor\}$ and a maximal antichain $\mathcal A\subseteq\binom{[n]}{l}\cup\binom{[n]}{l+1}$ with $\abs{\mathcal A}=m$.
\end{theorem}
The lower bound on $m$ in \Cref{thm:main_result} is best possible: In~\cite{Gruettmueller2009} it was shown that $\binom{n}{2}-\left\lfloor\frac{(n+1)^2}{8}\right\rfloor$ is the minimum size of a maximal antichain consisting of $2$-sets and $3$-sets, and with singletons and $2$-sets we only get the sizes $\binom{t}{2}+n-t$ for $t\in\{0,1,\dots,n\}$. Therefore, at least one of the two numbers $\binom{n}{2}-\left\lfloor\frac{(n+1)^2}{8}\right\rfloor-1$ and $\binom{n}{2}-\left\lfloor\frac{(n+1)^2}{8}\right\rfloor-2$ is not the size of a flat maximal antichain.\footnote{Usually $\binom{n}{2}-\left\lfloor\frac{(n+1)^2}{8}\right\rfloor-1$ is the largest missing value, but if $s=4(n-1)(n-2)-8\left\lfloor\frac{(n+1)^2}{8}\right\rfloor-7$ turns out to be a square then $\binom{n}{2}-\left\lfloor\frac{(n+1)^2}{8}\right\rfloor-1=\binom{t}{2}+n-t=$ with $t=\frac32+\frac12\sqrt{s}$ is the size of a flat maximal antichain consisting of 1-sets and 2-sets. For $n\leq \num{1000000}$ this happens only for the following 23 values of $n$: 5, 6, 7, 8, 9, 33, 48, 63, 78, 93, 429, 638, 847, 1056, 1265, 5945, 8856, \num{11767}, \num{14678}, \num{17589}, \num{82773}, \num{123318}, \num{163863}, \num{204408}, \num{244953}.}  Combining \Cref{thm:main_result,thm:MAC_sizes}, we obtain the following characterization of the sizes of flat maximal antichains.
\begin{corollary}\label{cor:mfac_sizes}
  Let $n$ and $m$ be integers with $1\leq m\leq\binom{n}{k}$ where $k=\lceil n/2\rceil$. There exists a flat maximal antichain of size $m$ in $B_n$ if and only if at least one of the following three conditions is satisfied:
  \begin{enumerate}[(i)]
  \item $\binom{n}{2}-\left\lfloor\frac{(n+1)^2}{8}\right\rfloor\leq m\leq\binom{n}{k}-k\left\lceil\frac{k+1}{2}\right\rceil$
  \item $m=\binom{n}{k}-tl+\binom{a}{2}+\binom{b}{2}+c$ for some integers $l\in\{k,n-k\}$, $t\in\{0,\dots,k\}$, $a\geq b\geq c\geq 0$, $a+b\leq t$.
  \item $m=\binom{t}{2}+n-t$ for some $t\in\{0,1,\dots,n\}$.
  \end{enumerate}
\end{corollary}
In particular, almost all sizes of maximal antichains are obtained by flat maximal antichains, the exceptions being the elements of the set $\left\{1,2,\dots,\binom{n}{2}-\left\lfloor\frac{(n+1)^2}{8}\right\rfloor-1\right\}\setminus\{\binom{t}{2}+n-t\,:\,t\in\{0,\dots,n\}\}$.

\subsection*{Related work}
In \cite{Gruettmueller2009} the problem of finding the smallest size of a maximal antichain in $\binom{n}{l}\cup\binom{n}{l+1}$ was raised and a complete answer was provided for $l=2$. A stability version of this result together with asymptotic bounds for general $l$ have been presented in \cite{Gerbner2012}. The more general question of minimizing the size of an antichain on two arbitrary levels, not necessarily consecutive, has been studied in \cite{Kalinowski2013}. Recently, there has been an increasing interest in poset saturation problems, asking for the possible sizes of posets which are maximal with respect to the property of not containing a fixed poset (see \cite{Keszegh_2021} for an overview). 

Another related line of research is the investigation of the domination number of the bipartite graph with vertex set $\binom{[n]}{l}\cup\binom{[n]}{l+1}$ (and more generally $\binom{[n]}{l}\cup\binom{[n]}{l+t}$ for $t\in\{1,\dots,n-l\}$) where adjacency is given by set inclusion \cite{Badakhshian_2019,Balogh2021}. A dominating set in a graph is a set $W$ of vertices with the property that every vertex outside $W$ has a neighbor in $W$, and the \emph{domination number} of a graph is the minimum size of a dominating set. Clearly, a maximal antichain $\mathcal A$ is a minimal dominating set which is also an independent set (that is, there are no edges between elements of $\mathcal A$). The problem of finding the minimum size of a flat maximal antichain can be viewed as asking for the \emph{independent domination number} of the graph: the minimum size of an independent dominating set. For $l=2$, the domination number is asymptotically equal to the independent domination number~\cite{Balogh2021}, and it is conjectured that the exact values are equal. 

\section{High level overview}\label{sec:high_level}
Before going into the technical details we provide a sketch of the main ideas behind the construction. This will be informal and without proofs. It turns out that this informal description is more straightforward in the upper half of $B_n$, that is, using sets of size at least $\lfloor n/2\rfloor$. The statement of \Cref{thm:main_result} and the formal proofs below take place in the lower half. Clearly, these two settings are equivalent via the observation that the family obtained by taking complements of the members of a maximal antichain is again a maximal antichain.  

For a set $A\in\binom{[n]}{l+1}$, its \emph{shadow} $\Delta A$ is the collection of all
$l$-subsets of $A$, that is, $\Delta A=\{A\setminus\{i\}\,:\,i\in A\}$. For a family $\mathcal
F\subseteq\binom{[n]}{l+1}$, the shadow $\Delta\mathcal F$ is union of the shadows of the members of
$\mathcal F$, that is, $\Delta\mathcal
F=\bigcup_{A\in\mathcal F}\Delta A=\{X\subseteq[n]\,:\,\abs{X}=l\text{ and }X\subseteq A\text{ for some }A\in\mathcal F\}$. This
concept is relevant because a flat maximal antichain $\mathcal
A\subseteq\binom{[n]}{l+1}\cup\binom{[n]}{l}$ is determined by its collection of $(l+1)$-sets:
$\mathcal A=\mathcal F\cup\binom{[n]}{l}\setminus\Delta\mathcal F$. Such an antichain is maximal if
and only if $\Delta A\not\subseteq\Delta\mathcal F$ for every $A\in\binom{[n]}{l+1}\setminus\mathcal
F$. 
Similarly, the \emph{shade} (or \emph{upper shadow}) $\nabla\mathcal G$ of $\mathcal G\subseteq\binom{[n]}{l}$ (with respect to the fixed ground set $[n]$) is the family of all $(l+1)$-subsets of $[n]$ which are a superset of some member of $\mathcal G$, i.e., 
  $\nabla\mathcal G = \{Y\,:\,Y\subseteq[n],\,\abs{Y}=l+1,\,B\subseteq Y  \text{ for some }B\in\mathcal G\}$.
  In verifying maximality of certain flat antichains constructed below it is convenient to notice that an antichain $\mathcal A=\mathcal A_1\cup\mathcal A_2$ with $\mathcal A_1\subseteq\binom{[n]}{l}$ and $\mathcal A_2\subseteq\binom{[n]}{l+1}$, is maximal if and only if $\mathcal A_1=\binom{[n]}{l}\setminus\Delta\mathcal A_2$ and $\mathcal A_2 = \binom{[n]}{l+1}\setminus\nabla\mathcal A_1$.

Given a positive integer $t$, the Kruskal-Katona Theorem \cite{Kruskal1963,Katona1968} says that among the families $\mathcal F\subseteq\binom{[n]}{l+1}$ with $\abs{\mathcal{F}}=t$, the shadow is minimized by taking the first $t$ sets in \emph{squashed order} (also known as \emph{colexicographic order}), where $A$ precedes $B$ in this order if $\max((A\cup B)\setminus(A\cap B))\in B$. The theorem also specifies the size of the shadow: Writing $t$ in the cascade representation $t=\binom{a_{l+1}}{l+1}+\binom{a_l}{l}+\dots+\binom{a_r}{r}$ with $a_{l+1}>a_l>\dots>a_r\geq r\geq 1$, the shadow of the first $t$ sets of size $l+1$ in squashed order has size $\binom{a_{l+1}}{l}+\binom{a_l}{l-1}+\dots+\binom{a_r}{r-1}$. 

We call a flat antichain $\mathcal A$ \emph{squashed} if it has the form $\mathcal A=\mathcal F\cup\binom{[n]}{l}\setminus\Delta\mathcal F$, where $\mathcal F$ is the collection of the first $\abs{\mathcal F}$ elements of $\binom{[n]}{l+1}$ in squashed order. Note that such a squashed flat antichain is a maximal antichain if and only if $r\geq 2$ (see \cite[Proposition 3.1]{Griggs2021} for a formal proof). The sizes of maximal squashed flat antichains have been studied in \cite{Griggs2021}, and they form the backbone of our construction.

The main idea in the proof of \Cref{thm:main_result} is to begin with the families of squashed flat maximal antichains on levels $l+1$ and $l$, and consider the gaps in the corresponding collection of sizes. The ultimate aim is to show that each gap value is attained by a flat maximal antichain on levels $i+1$ and $i$ for some $i$.
The difference of any two consecutive sizes of squashed flat maximal antichains on levels $l+1$ and
$l$ is at most $l$, and the idea is to fill the gaps between them by modifying the last few sets in
an initial segment $\mathcal F$ of $\binom{[n]}{l+1}$. More precisely, if $\mathcal A=\mathcal
F\cup\binom{[n]}{l}\setminus\Delta\mathcal F$ is a squashed flat maximal antichain of size $m$, and
$\delta\leq l$, we change the last few sets of $\mathcal F$ to obtain a family $\mathcal F'$ such
that $\mathcal A'=\mathcal F'\cup\binom{[n]}{l}\setminus\Delta\mathcal F'$ is a maximal antichain of
size $m-\delta$. We show that this is possible by an induction argument: The base case is $l=n-3$, and for $l\leq n-4$ recursive constructions are used.   

\subsection*{The base case \texorpdfstring{$l=n-3$}{}} 
We start by constructing, for every  $i\in\{0,1,\dots,\lambda\}$, a family $\mathcal F_i\subseteq\binom{[n]}{l+1}$ with $\abs{\mathcal F_i}=\left\lceil\frac{l+1}{2}\right\rceil$ and $\abs{\Delta\mathcal F_i}=(l+1)\left\lceil\frac{l+1}{2}\right\rceil-i$ with the property that $n\in A$ for every $A$ in every $\mathcal F_i$. Here $\lambda$ is a parameter depending on $n$: $\lambda=f(n-2)$ in \Cref{prop:large_flat} below, in particular $\lambda=l$ for $n\geq 13$. For $t=0,1,\dots,l$, we set $\mathcal F_{i,t}=\mathcal F_i\cup\{A_1,\dots,A_t\}$ where $A_j$ is the $j$-th $(l+1)$-set in squashed order (note that $n\not\in A_j$ for all $j$), and $\mathcal A_{i,t}=\mathcal F_{i,t}\cup\binom{[n]}{l}\setminus\Delta\mathcal F_{i,t}$. The families $\mathcal F_i$ are constructed in such a way that, for all $i$ and $t$, adding $A_t$ to $\mathcal A_{i,t-1}$ increases the shadow by $l+1-t$ for $t\leq\lceil(n-2)/2\rceil$ and by $l+2-t$ if $t>\lceil(n-2)/2\rceil$. As a consequence, the sizes of the maximal antichains $\mathcal A_{i,t}$ cover a certain interval ending at $\abs{\mathcal A_{\lambda,0}}=\binom{n}{3}-l\left\lceil\frac{l+1}{2}\right\rceil+\lambda$. By an induction argument we also have an interval of sizes starting at $\binom{n}{l+1}-5$ and ending above the starting point of the previously mentioned interval. Combining these two intervals, we have an interval of sizes from $\binom{n}{2}-5$ to $\abs{\mathcal A_{\lambda,0}}$ (see \Cref{lem:base_case}). For example, for $(n,l)=(13,10)$, the families $\mathcal F_0,\dots,\mathcal F_{10}$ are illustrated in \Cref{fig:graphs_13_11_a,fig:graphs_13_11_b,fig:graphs_13_11_c,fig:graphs_13_11_d}. Here we use that we can think of a set of size $n-2$ as the complement of an edge of a graph, and a family of sets of size $n-2$ can be represented by a graph. In this sense, the members of $\mathcal F_i$ are the complements of the edges of the corresponding graph, for instance
\[\mathcal F_0=\left\{[13]\setminus\{1,12\},\,[13]\setminus\{2,11\},\,[13]\setminus\{3,10\},\,[13]\setminus\{4,9\},\,[13]\setminus\{5,8\},\,[13]\setminus\{6,7\}\right\}.\]
\begin{figure}[htb]
  \begin{minipage}{.32\linewidth}
    \centering
    \begin{tikzpicture}
      \foreach \a in {1,2,...,13} {
        \node[circle,fill=black,outer sep=1pt,inner sep=1pt,label=\a*360/13:$\a$] (\a) at (\a*360/13:1.5cm) {};
      }
      \draw[thick] (12) -- (1);
      \draw[thick] (11) -- (2);
      \draw[thick] (10) -- (3);
      \draw[thick] (9) -- (4);
      \draw[thick] (8) -- (5);
      \draw[thick] (7) -- (6);
    \end{tikzpicture}
  \end{minipage}\hfill
  \begin{minipage}{.32\linewidth}
    \centering
    \begin{tikzpicture}
      \foreach \a in {1,2,...,13} {
        \node[circle,fill=black,outer sep=1pt,inner sep=1pt,label=\a*360/13:$\a$] (\a) at (\a*360/13:1.5cm) {};
      }
      \draw[thick] (12) -- (11);
      \draw[thick] (11) -- (2);
      \draw[thick] (10) -- (3);
      \draw[thick] (9) -- (4);
      \draw[thick] (8) -- (5);
      \draw[thick] (7) -- (6);
    \end{tikzpicture}
  \end{minipage}\hfill
  \begin{minipage}{.32\linewidth}
    \centering
    \begin{tikzpicture}
      \foreach \a in {1,2,...,13} {
        \node[circle,fill=black,outer sep=1pt,inner sep=1pt,label=\a*360/13:$\a$] (\a) at (\a*360/13:1.5cm) {};
      }
      \draw[thick] (12) -- (11);
      \draw[thick] (11) -- (10);
      \draw[thick] (10) -- (3);
      \draw[thick] (9) -- (4);
      \draw[thick] (8) -- (5);
      \draw[thick] (7) -- (6);
    \end{tikzpicture}
  \end{minipage}
  \caption{The edges of the graphs are the complements of the members of $\mathcal F_0$, $\mathcal F_1$, $\mathcal F_2$.}\label{fig:graphs_13_11_a}
\end{figure}
\begin{figure}[htb]
  \begin{minipage}{.32\linewidth}
    \centering
    \begin{tikzpicture}
      \foreach \a in {1,2,...,13} {
        \node[circle,fill=black,outer sep=1pt,inner sep=1pt,label=\a*360/13:$\a$] (\a) at (\a*360/13:1.5cm) {};
      }
      \draw[thick] (12) -- (10);
      \draw[thick] (11) -- (10);
      \draw[thick] (10) -- (3);
      \draw[thick] (9) -- (4);
      \draw[thick] (8) -- (5);
      \draw[thick] (7) -- (6);
    \end{tikzpicture}
  \end{minipage}\hfill
  \begin{minipage}{.32\linewidth}
    \centering
    \begin{tikzpicture}
      \foreach \a in {1,2,...,13} {
        \node[circle,fill=black,outer sep=1pt,inner sep=1pt,label=\a*360/13:$\a$] (\a) at (\a*360/13:1.5cm) {};
      }
      \draw[thick] (12) -- (10);
      \draw[thick] (11) -- (10);
      \draw[thick] (10) -- (9);
      \draw[thick] (9) -- (4);
      \draw[thick] (8) -- (5);
      \draw[thick] (7) -- (6);
    \end{tikzpicture}
  \end{minipage}\hfill
  \begin{minipage}{.32\linewidth}
    \centering
    \begin{tikzpicture}
      \foreach \a in {1,2,...,13} {
        \node[circle,fill=black,outer sep=1pt,inner sep=1pt,label=\a*360/13:$\a$] (\a) at (\a*360/13:1.5cm) {};
      }
      \draw[thick] (12) -- (10);
      \draw[thick] (11) -- (10);
      \draw[thick] (10) -- (9);
      \draw[thick] (9) -- (8);
      \draw[thick] (8) -- (5);
      \draw[thick] (7) -- (6);
    \end{tikzpicture}
  \end{minipage}
  \caption{The edges of the graphs are the complements of the members of $\mathcal F_3$, $\mathcal F_4$, $\mathcal F_5$.}\label{fig:graphs_13_11_b}
\end{figure}
\begin{figure}[htb]
  \begin{minipage}{.32\linewidth}
    \centering
    \begin{tikzpicture}
      \foreach \a in {1,2,...,13} {
        \node[circle,fill=black,outer sep=1pt,inner sep=1pt,label=\a*360/13:$\a$] (\a) at (\a*360/13:1.5cm) {};
      }
      \draw[thick] (12) -- (9);
      \draw[thick] (11) -- (9);
      \draw[thick] (10) -- (9);
      \draw[thick] (9) -- (4);
      \draw[thick] (8) -- (5);
      \draw[thick] (7) -- (6);
    \end{tikzpicture}
  \end{minipage}\hfill
  \begin{minipage}{.32\linewidth}
    \centering
    \begin{tikzpicture}
      \foreach \a in {1,2,...,13} {
        \node[circle,fill=black,outer sep=1pt,inner sep=1pt,label=\a*360/13:$\a$] (\a) at (\a*360/13:1.5cm) {};
      }
      \draw[thick] (12) -- (9);
      \draw[thick] (11) -- (9);
      \draw[thick] (10) -- (9);
      \draw[thick] (9) -- (8);
      \draw[thick] (8) -- (5);
      \draw[thick] (7) -- (6);
    \end{tikzpicture}
  \end{minipage}\hfill
  \begin{minipage}{.32\linewidth}
    \centering
    \begin{tikzpicture}
      \foreach \a in {1,2,...,13} {
        \node[circle,fill=black,outer sep=1pt,inner sep=1pt,label=\a*360/13:$\a$] (\a) at (\a*360/13:1.5cm) {};
      }
      \draw[thick] (12) -- (9);
      \draw[thick] (11) -- (9);
      \draw[thick] (10) -- (9);
      \draw[thick] (9) -- (8);
      \draw[thick] (8) -- (7);
      \draw[thick] (7) -- (6);
    \end{tikzpicture}
  \end{minipage}
  \caption{The edges of the graphs are the complements of the members of $\mathcal F_6$, $\mathcal F_7$, $\mathcal F_8$.}\label{fig:graphs_13_11_c}
\end{figure}
\begin{figure}[htb]
  \begin{minipage}{.49\linewidth}
    \centering
    \begin{tikzpicture}
      \foreach \a in {1,2,...,13} {
        \node[circle,fill=black,outer sep=1pt,inner sep=1pt,label=\a*360/13:$\a$] (\a) at (\a*360/13:1.5cm) {};
      }
      \draw[thick] (12) -- (9);
      \draw[thick] (11) -- (9);
      \draw[thick] (10) -- (9);
      \draw[thick] (9) -- (7);
      \draw[thick] (8) -- (7);
      \draw[thick] (7) -- (6);
    \end{tikzpicture}
  \end{minipage}\hfill
  \begin{minipage}{.49\linewidth}
    \centering
    \begin{tikzpicture}
      \foreach \a in {1,2,...,13} {
        \node[circle,fill=black,outer sep=1pt,inner sep=1pt,label=\a*360/13:$\a$] (\a) at (\a*360/13:1.5cm) {};
      }
      \draw[thick] (12) -- (8);
      \draw[thick] (11) -- (8);
      \draw[thick] (10) -- (8);
      \draw[thick] (9) -- (8);
      \draw[thick] (8) -- (5);
      \draw[thick] (7) -- (6);
    \end{tikzpicture}
  \end{minipage}
  \caption{The edges of the graphs are the complements of the members of $\mathcal F_9$ and $\mathcal F_{10}$.}\label{fig:graphs_13_11_d}
\end{figure}

Adding the first $t$ $11$-sets in squashed order corresponds to adding the edges $\{13,12\},\,\{13,11\},\dots,\{13,13-t\}$ to these graphs, and as a result we obtain the following maximal antichains:
\begin{itemize}
    \item $\mathcal{A}_{i,0}$, $i=0,1,\dots,10$, each with $6$ $11$-sets, cover the sizes $\binom{13}{10}-60=226,\dots,\binom{13}{10}-50=236$,
    \item $\mathcal{A}_{i,1}$, $i=0,1,\dots,10$, each with $7$ $11$-sets, cover the sizes  $\binom{13}{10}-69=217,\dots,\binom{13}{10}-59=227$,
    \item $\mathcal{A}_{i,2}$, $i=0,1,\dots,10$, each with $8$ $11$-sets, cover the sizes $\binom{13}{10}-77=209,\dots,\binom{13}{10}-67=219$,
    \item $\mathcal{A}_{i,3}$, $i=0,1,\dots,10$, each with $9$ $11$-sets, cover the sizes $\binom{13}{10}-84=202,\dots,\binom{13}{10}-74=212$,
    \item $\mathcal{A}_{i,4}$, $i=0,1,\dots,10$, each with $10$ $11$-sets, cover the sizes $\binom{13}{10}-90=196,\dots,\binom{13}{10}-80=206$,
    \item $\mathcal{A}_{i,5}$, $i=0,1,\dots,10$, each with $11$ $11$-sets, cover the sizes $\binom{13}{10}-95=191,\dots,\binom{13}{10}-85=201$,
    \item $\mathcal{A}_{i,6}$, $i=0,1,\dots,10$, each with $12$ $11$-sets, cover the sizes $\binom{13}{10}-99=187,\dots,\binom{13}{10}-89=197$,
    \item $\mathcal{A}_{i,7}$, $i=0,1,\dots,10$, each with $13$ $11$-sets, cover the sizes $\binom{13}{10}-103=183,\dots,\binom{13}{10}-93=193$,
    \item $\mathcal{A}_{i,8}$, $i=0,1,\dots,10$, each with $14$ $11$-sets, cover the sizes $\binom{13}{10}-106=180,\dots,\binom{13}{10}-96=190$,
    \item $\mathcal{A}_{i,9}$, $i=0,1,\dots,10$, each with $15$ $11$-sets, cover the sizes $\binom{13}{10}-108=178,\dots,\binom{13}{10}-98=188$,
    \item $\mathcal{A}_{i,10}$, $i=0,1,\dots,10$, each with $16$ $11$-sets, cover the sizes $\binom{13}{10}-109=177,\dots,\binom{13}{10}-99=187$.
\end{itemize}
Overall we get maximal antichains of all sizes in the interval $[\binom{13}{10}-109,\binom{13}{10}-50]=[177,236]$ (illustrated by the plot on the left in \Cref{fig:example_13_11}). 
\begin{figure}[htb]
    \begin{minipage}{.49\textwidth}
    \includegraphics[width=\textwidth]{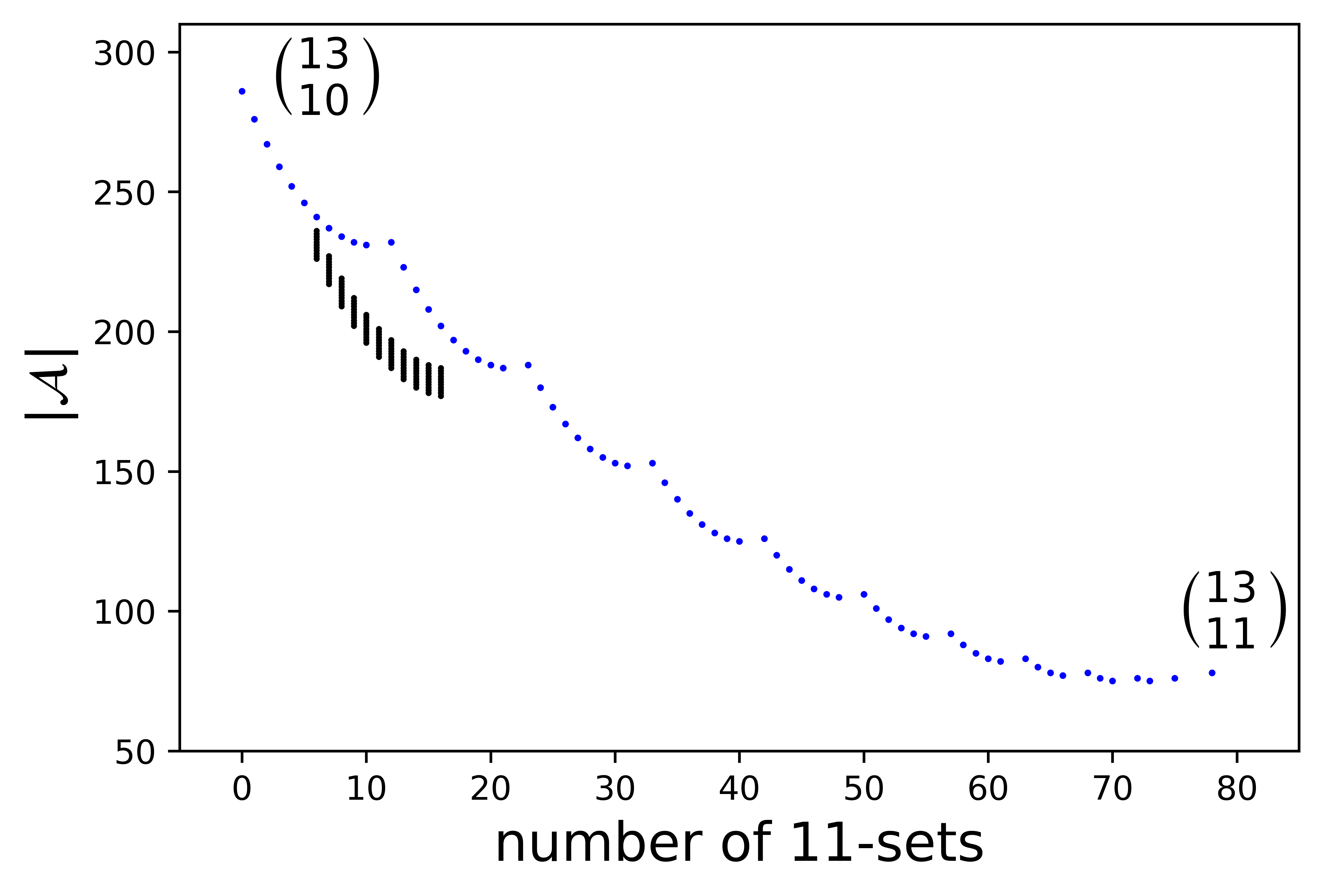}    
    \end{minipage}\hfill
    \begin{minipage}{.49\textwidth}
    \includegraphics[width=\textwidth]{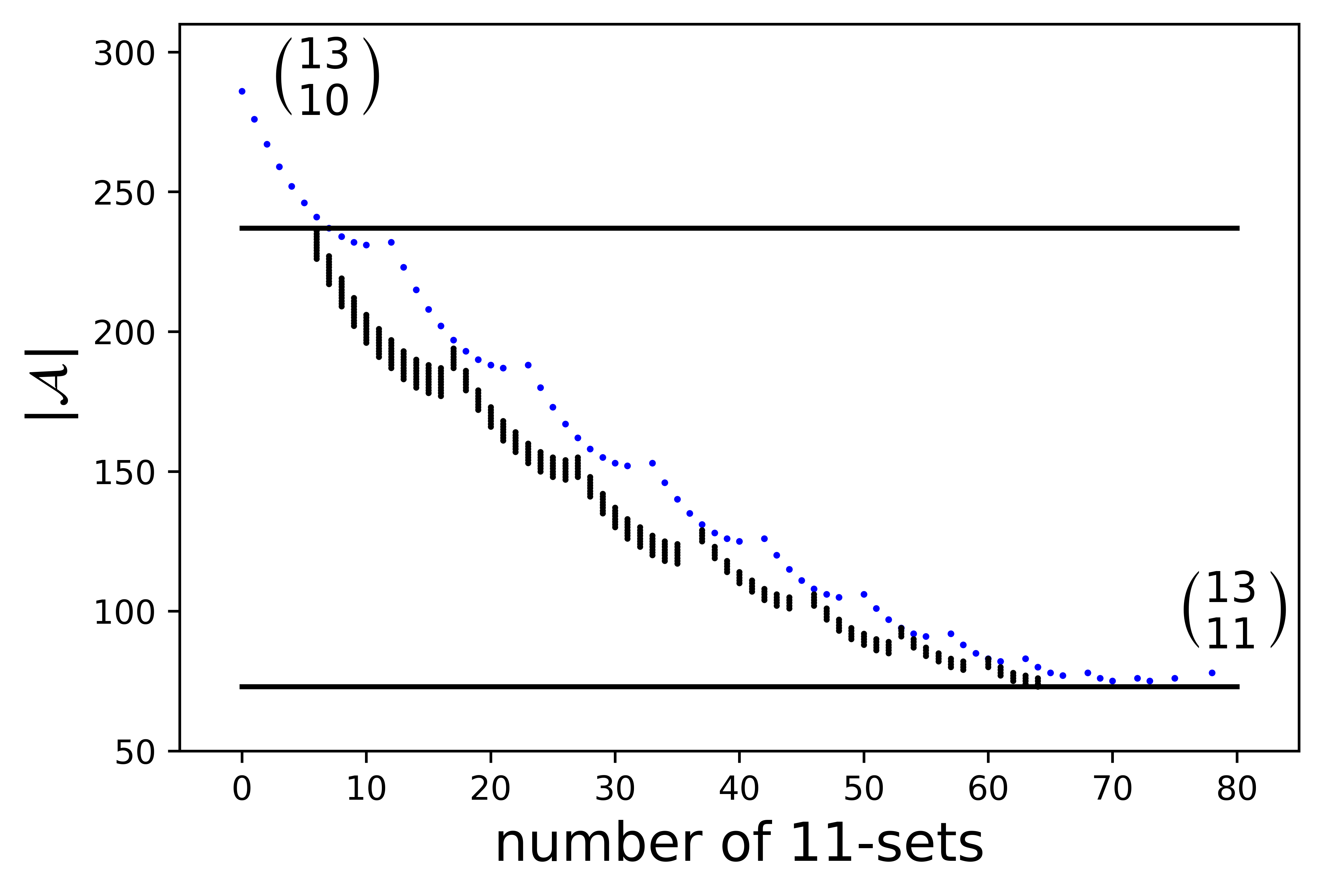}    
    \end{minipage}
    \caption{The dots correspond to the squashed maximal flat antichains $\mathcal{A}\subseteq\binom{[13]}{10}\cup\binom{[13]}{11}$. In the left plot, the segments correspond to the maximal antichains $\mathcal A_{i,t}$ ($i=0,1,\dots,10$, $t=0,1,\dots,10$). In the right plot, we have also added the maximal flat antichains obtained by recursion from $n=12$. The interval $[73,237]$ of sizes of maximal antichains contained in $\binom{[13]}{11}\cup\binom{[13]}{10}$ is indicated by horizontal lines.}\label{fig:example_13_11}
\end{figure}
By induction, we also have maximal antichains in $\binom{[12]}{9}\cup\binom{[12]}{10}$ of all sizes between $\binom{12}{2}-5=61$ and $\binom{12}{3}-38=182$. For each such antichain $\mathcal A$ we obtain a maximal antichain $\mathcal A'\subseteq\binom{[13]}{10}\cup\binom{[13]}{11}$ with $\abs{A'}=\abs{A}+12$ by (1) adding the element $13$ to every $A\in\mathcal A$ and (2) adding the 11-sets $[12]\setminus\{i\}$ for $i=1,2,\dots,12$ to the antichain (\Cref{lem:recursions}(i)). This yields the sizes from $\binom{13}{2}-5=73$ to $182+12>177$. Thus, we have maximal antichains of all sizes between $\binom{13}{2}-5$ and $\binom{13}{3}-50$ (see the right hand side plot in \Cref{fig:example_13_11}). 

\subsection*{The case \texorpdfstring{$l\leq n-4$}{}} 
For $t=0,1,\dots,\binom{n}{l+1}$, let $\mathcal I_t$ be the collection of the first $t$ elements of $\binom{[n]}{l+1}$ in squashed order, and define the function $f:\{0,1,\dots,\binom{n}{l+1}\}\to\mathbb{N}$ by $f(t)=t+\binom{n}{l}-\abs{\Delta\mathcal I_t}$. In other words, $f(t)$ is the size of the squashed full flat antichain with $t$ $(l+1)$-sets. For $t$ with cascade representation $t=\binom{a_{l+1}}{l+1}+\binom{a_l}{l}+\dots+\binom{a_r}{r}$ with $a_r\geq r+2$, we can use the base case to obtain antichains of all sizes between $f(t+\binom{r+1}{r-1})-5$ and a value just below $f(t)$. The base construction provides, for every $\delta\in\{\delta_0,\delta_0+1,\dots,\binom{r+1}{r-2}-\binom{r+1}{r-1}+5\}$ a family $\mathcal F'\subseteq\binom{[r+1]}{r-1}$ with  $\abs{\Delta\mathcal{F}'}-\abs{\mathcal{F}'}=\delta$, where $\delta_0<r^2/2$ (see \Cref{lem:base_case} for the precise value). The maximal antichain $\mathcal A=\mathcal F\cup\binom{[n]}{l}\setminus\Delta\mathcal{F}$ with $\mathcal F=\mathcal I_t\cup\left\{A\cup\{a_j+1\,:\,r\leq j\leq l+1\}\,:\,A\in\mathcal F'\right\}$ has size $\abs{\mathcal A}=f(t)-\delta$ since we have the partition $\Delta\mathcal F=\Delta\mathcal I_t\cup\left\{A\cup\{a_j+1\,:\,r\leq j\leq l+1\}\,:\,A\in\Delta\mathcal F'\right\}$, hence
\[\abs{\mathcal A}=\abs{\mathcal F}+\binom{n}{l}-\abs{\Delta\mathcal F}=t+\abs{\mathcal F'}+\binom{n}{l}-\abs{\Delta\mathcal I_t}-\abs{\Delta\mathcal F'}=f(t)-\delta.\]

For example, for $n=19$, $l=12$ and $t=\binom{18}{13}+\binom{15}{12}$, we use the base case for $\binom{[13]}{10}\cup\binom{[13]}{11}$ to obtain maximal antichains $\mathcal A\subseteq\binom{[19]}{12}\cup\binom{[19]}{13}$ of all sizes between $f(t+\binom{13}{11})-5$ and $f(t)-49$. The base case provides, for every $\delta\in\{49,\dots,\binom{13}{10}-\binom{13}{11}+5\}$ a family $\mathcal F'\subseteq\binom{[13]}{11}$ with $\abs{\Delta\mathcal F'}-\abs{\mathcal F'}=\delta$, and then the maximal antichain  $\mathcal A=\mathcal F\cup\binom{[19]}{12}\setminus\Delta\mathcal{F}$ for
\[\mathcal F=\binom{[18]}{13}\cup\left\{A\cup\{19\}\,:\,A\in\binom{[15]}{12}\right\}\cup\left\{A\cup\{16,19\}\,:\,A\in\mathcal F'\right\}\]
has size $\abs{\mathcal A}=f(t)-\delta$ (see \Cref{fig:example_19_13} for an illustration).
\begin{figure}[htb]
    \begin{minipage}{.49\textwidth}
    \includegraphics[width=\textwidth]{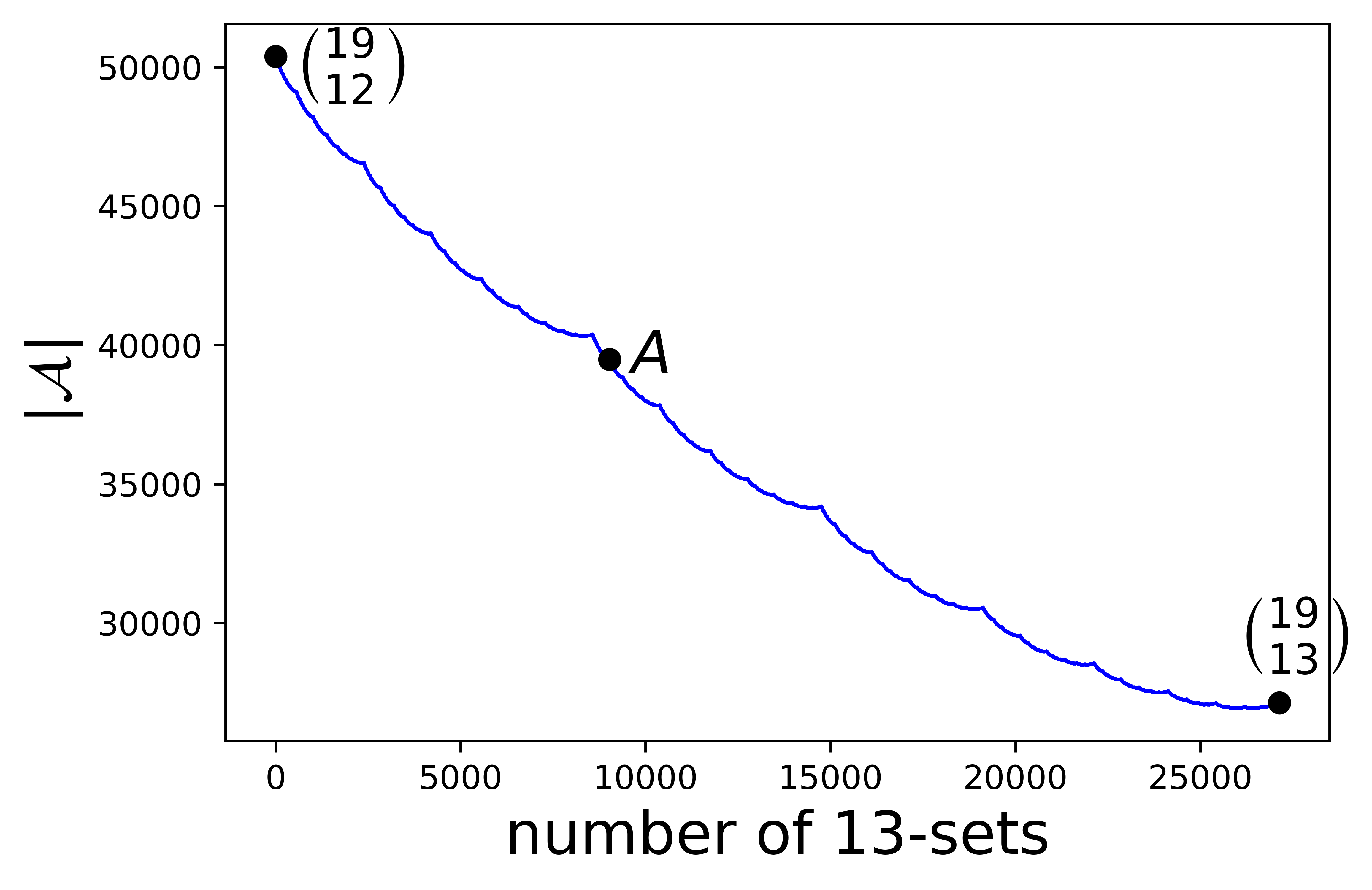}    
    \end{minipage}\hfill
    \begin{minipage}{.49\textwidth}
    \includegraphics[width=\textwidth]{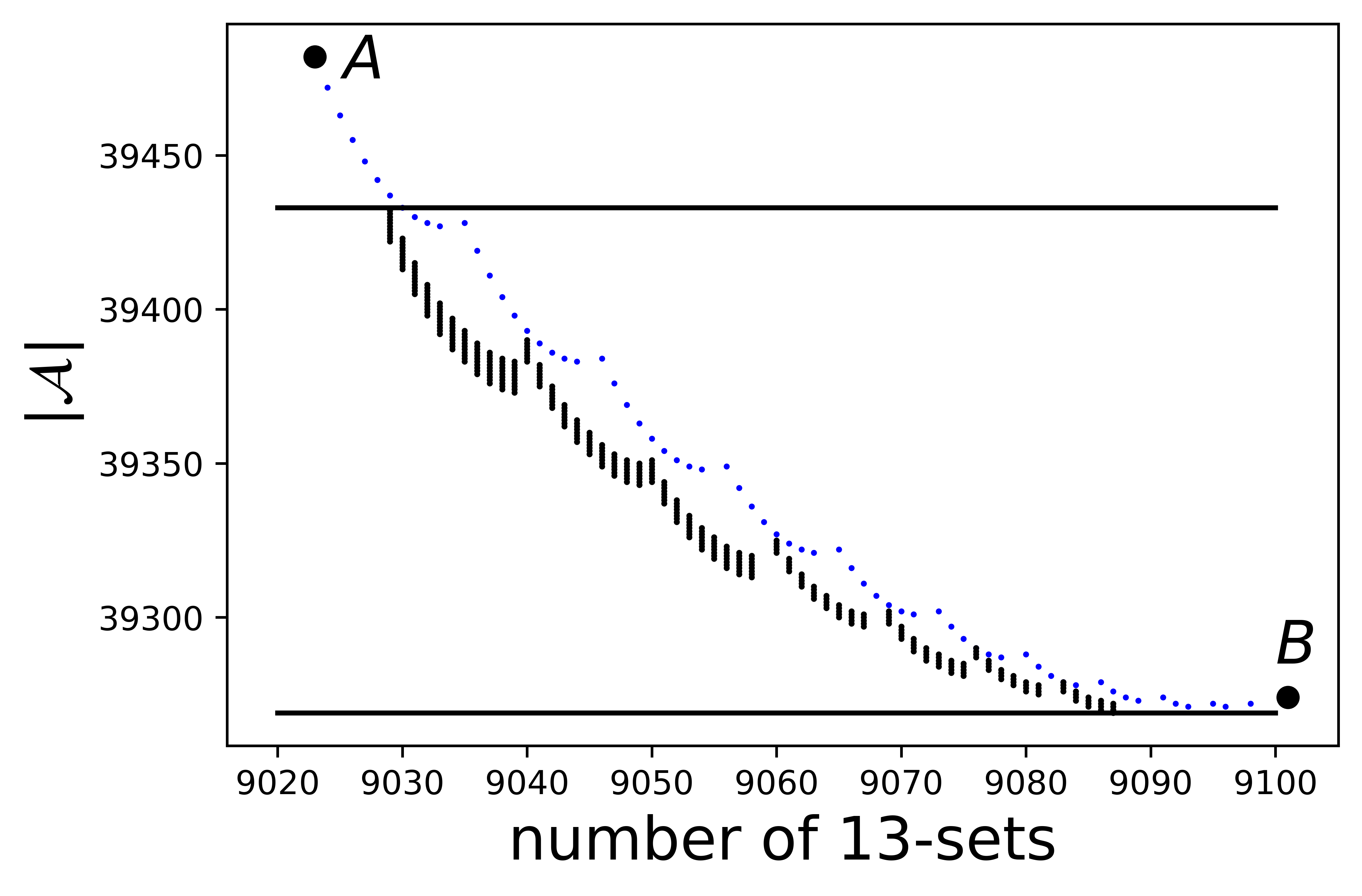}    
    \end{minipage}
    \caption{Constructing an interval of sizes of maximal antichains in $\binom{[19]}{12}\cup\binom{19}{13}$. The labels on the dots are the sizes of the corresponding maximal flat antichains. The part of the graph shown in the plot on the right has the same structure as the graph for the sizes of the maximal flat antichains in $\binom{[13]}{10}\cup\binom{[13]}{11}$ (see \Cref{fig:example_13_11}). The point $A$ corresponds to a maximal antichain of size $\binom{19}{12}+\binom{18}{13}-\binom{18}{12}+\binom{15}{12}-\binom{15}{11}=\num{39482}$, and the point $B$ represents size $A+\binom{13}{11}-\binom{13}{10}=\num{39274}$. The interval $[\num{39269},\num{39433}]$ of sizes of maximal antichains contained in $\binom{[19]}{13}\cup\binom{[19]}{12}$ is indicated by horizontal lines.}\label{fig:example_19_13}
\end{figure}

According to the description above, we obtain maximal antichains $\mathcal A\subseteq\binom{[n]}{l}\cup\binom{[n]}{l+1}$ whose sizes form many short intervals inside the large interval that goes roughly from $\binom{n}{l+1}$ to $\binom{n}{l}$, but there are many gaps between those intervals. In order to avoid these gaps, we proceed a bit more carefully when the base case is transferred: Instead of jumping directly from $l=n-3$ to an arbitrary $l$, we move recursively from level to level (using the recursions given in \Cref{lem:recursions}). Keeping careful track of the constructed intervals, it turns out that we can push the lower end of the interval in each level of the recursion a bit further down, and, as a consequence, the recursively constructed sizes form a single interval. This is the content of \Cref{prop:large_flat} which is proved in \Cref{sec:large_flat}. Finally, the gaps between these intervals for varying $l$ are closed in \Cref{sec:proof_main_result}.    

\medskip

For the rest of the paper we work in the lower half of $B_n$, that is, with subsets of size at most $\lceil n/2\rceil$. The main reason is that a major part of the argument (\Cref{subsec:k_2}) is about the base case which deals with maximal antichains in $\binom{[n]}{n-2}\cup\binom{[n]}{n-3}$ or in $\binom{[n]}{2}\cup\binom{[n]}{3}$. Working with sets of size $n-2$ and $n-3$ was convenient for the informal overview above, because the $(n-2)$-sets are added in squashed order, and this ordering is very familiar for many people working in the field. For our formal proof it is easier to work with sets of size 2 and 3, because this allows us to interpret the 2-sets in the antichain directly as the edges of a graph (without the need to take complements), and using concepts from graph theory turns out to be useful for our argument. Note that the plots in \Cref{fig:example_13_11,fig:example_19_13} look essentially the same if we work in the lower half: We just have to change the labels of the $x$-axes to ``number of $2$-sets'' and ``number of $6$-sets'', respectively, and replace $\binom{19}{12}$ with $\binom{19}{7}$ etc.        

For notational convenience, we reserve the letter $k$ for $\lceil n/2\rceil$, so that $n\in\{2k-1,2k\}$.

\subsection*{Outline of the proof}
Let $S(n,l)$ be the subset of $S(n)$ containing the sizes of flat antichains on levels $l$ and~$l+1$:
\[S(n,l)=\left\{\left\lvert\mathcal A\right\rvert\,:\,\mathcal A\subseteq\binom{[n]}{l}\cup\binom{[n]}{l+1}\text{ is a maximal antichain in }B_n\right\}.\]
The following proposition, to be proved in \Cref{sec:large_flat,sec:proof_main_result}, states that $S(n,l)$ covers an interval from slightly below $\binom{n}{l}$ to slightly below $\binom{n}{l+1}$. For convenience, we denote by $C_l$ the sum of the first $l$ Catalan numbers, that is, $C_l=\sum_{i=1}^l\frac{1}{i+1}\binom{2i}{i}$.
\begin{proposition}\label{prop:flat}
$\phantom{x}$
\begin{enumerate}[(i)]
\itemsep+2ex
  \item
  $\displaystyle\left[\binom{n}{2}-\left\lfloor\frac{(n+1)^2}{8}\right\rfloor,\ \binom{n}{3}-(n-3)\left\lceil\frac{n-2}{2}\right\rceil\right]\subseteq S(n,2)$
  \item
  $\displaystyle \left[\binom{n}{l}-\gamma,\ \binom{n}{l+1}-(n-l-1)\left\lceil\frac{n-l}{2}\right\rceil\right]\subseteq S(n,l)$
  ~for~ $\displaystyle 3\leq l\leq\frac{n-2}{2}$, \\[2ex]
  where $\displaystyle\gamma=\max\left\{3+C_l,\,(1+(l-1)(n-k))\left(\left\lceil\frac{1}{k}\binom{k}{l}\right\rceil-l+2\right)\right\}$
\end{enumerate}
\end{proposition}
It has been shown in \cite{Griggs2021} that the minimum size of a squashed maximal flat antichain on levels $l$ and $l+1$ is $\binom{n}{l}-C_l$. The first options for $\gamma$ in the second part of \Cref{prop:flat} says that we can make it smaller by 3 if we drop the squashed condition.\footnote{The flat maximal antichain of size $\binom{n}{l}-C_l-3$ which comes out of our proof of \Cref{prop:flat} has the following $(l+1)$-sets: the first $2+\sum_{i=4}^{l+1}\binom{2i-2}{i}$ sets in squashed order, and the two sets $\{3,5,6\}\cup\{7,9,\dots,2l+1\}$ and $\{4,5,6\}\cup\{7,9,\dots,2l+1\}$.}
The right end of the interval corresponds to an antichain obtained by taking $\left\lceil\frac{n-l}{2}\right\rceil$ $l$-sets such that no two of them have a common superset of size $l+1$, together with all $(l+1)$-sets which are not a superset of any of the chosen $l$-sets.

For $n\geq 15$, \Cref{prop:flat} is already sufficient to deduce the statement of \Cref{thm:main_result}. For instance, for $n=15$ the intervals obtained from \Cref{prop:flat} are  
\begin{itemize}
    \item $\left[\binom{15}{2}-32,\,\binom{15}{3}-84\right]=[73,371]\subseteq S(15,2)$,
    \item $\left[\binom{15}{3}-90,\,\binom{15}{4}-66\right]=[365,1299]\subseteq S(15,3)$,
    \item $\left[\binom{15}{4}-154,\,\binom{15}{5}-60\right]=[1211,2943]\subseteq S(15,4)$,
    \item $\left[\binom{15}{5}-116,\,\binom{15}{6}-45\right]=[2887,4960]\subseteq S(15,5)$,
    \item $\left[\binom{15}{6}-199,\,\binom{15}{7}-40\right]=[4860,6395]\subseteq S(15,6)$.
\end{itemize}
The union of these intervals is $\left[\binom{15}{2}-32,\binom{15}{7}-40\right]=[73,6395]$ which is the one required by \Cref{thm:main_result}. 

The proof of \Cref{prop:flat} comes in two parts corresponding to the two possible values of $\gamma$. In \Cref{sec:large_flat} we construct maximal antichains $\mathcal A\subseteq\binom{[n]}{l}\cup\binom{[n]}{l+1}$ for all sizes $m$ with
\[\binom{n}{l}-3-\sum_{i=1}^{l}\frac{1}{i+1}\binom{2i}{i}\leq m\leq\binom{n}{l+1}-(n-l-1)\left\lceil\frac{n-l}{2}\right\rceil,\]
and in Section~\ref{subsec:thm_flat} we use a result from coding theory to obtain the sizes $m$ with
\[\binom{n}{l}-(1+(l-1)(n-k))\left(\left\lceil\frac{1}{k}\binom{k}{l}\right\rceil-l+2\right)\leq m\leq\binom{n}{l}-(l-1)^2-1.\]
Combining these two results, we obtain \Cref{prop:flat}. The proof of \Cref{thm:main_result} is completed in \Cref{subsec:proof_main_thm} by verifying that the intervals for $l=2,3,\dots,\left\lfloor\frac{n-2}{2}\right\rfloor$ in \Cref{prop:flat}, together with some extra constructions for small $n$, cover the required interval.


\section{Large maximal flat antichains}\label{sec:large_flat}
In this section we prove the following half of \Cref{prop:flat} (one of the two possible values of $\gamma$).
\begin{proposition}\label{prop:large_flat}
   For every integer $l$ with $2\leq l\leq(n-2)/2$, 
  \[\left[\binom{n}{l}-3-C_l,\ \binom{n}{l+1}-(n-l-1)\left\lceil\frac{n-l}{2}\right\rceil+f(n-l)\right]\subseteq S(n,l),\]
  where the function $f:\nats\to\nats$ is given by
  \[f(t)=
    \begin{cases}
      0 & \text{for }t\leq 2,\\
      1 & \text{for }t\in\{3,4\},\\
      3 & \text{for }t\in\{5,6\},\\
      4 & \text{for }t\in\{7,8\},\\
      7 & \text{for }t\in\{9,10\},\\
      t-1 & \text{for }t\geq 11.
    \end{cases}
    \]
\end{proposition}
The interval in \Cref{prop:large_flat} goes a bit further to the right than needed for \Cref{prop:flat}. The proof is by induction using the recursions provided in \Cref{subsec:recurrences}, and the purpose of the function~$f$ is to facilitate the induction step by ensuring that certain shorter intervals of sizes overlap. Another point where the slight extension of the interval to the right turns out to be useful is the proof of \Cref{lem:small_l} where it reduces the number of small values of $(n,l)$ that have to be treated by an alternative method (see \Cref{tab:intervals}).

\Cref{subsec:k_2} contains the base case $l=2$. For this, we exploit an observation from~\cite{Gruettmueller2009} about a correspondence between maximal antichains $\mathcal A\subseteq\binom{[n]}{2}\cup\binom{[n]}{3}$ and graphs on $n$ vertices with the property that every edge is contained in at least one triangle. \Cref{subsec:k_geq_3} contains the induction step for the induction on $l$ and $n$ to conclude the proof of \Cref{prop:large_flat}.
\subsection{Recurrences}\label{subsec:recurrences}
In this subsection we prove \Cref{lem:recursions} which contains three recurrence relations between the sets $S(n,l)$. The ideas for these recursions is that we obtain maximal antichains in $\binom{[n]}{l}\cup\binom{[n]}{l+1}$ by any of the following three operations:
\begin{enumerate}
\item Start with a maximal antichain in $\binom{[n-1]}{l}\cup\binom{[n-1]}{l+1}$ and add all
  $l$-subsets of $[n]$ which contain the element $n$.
\item Start with a maximal antichain in $\binom{[n-1]}{l-1}\cup\binom{[n-1]}{l}$, add the element
  $n$ to each of its members, and add all $(l+1)$-subsets of $[n-1]$ to the resulting family.
\item Start with two maximal antichain $\mathcal A'$ and $\mathcal A''$ in
  $\binom{[n-2]}{l-1}\cup\binom{[n-2]}{l}$. Add the element $n-1$ to the members of
  $\mathcal A'$, and the element $n$ to the members of $\mathcal A''$. Take the union of the two
  resulting families and add all $(l+1)-$subsets of $[n-2]$ together with all the $l$-subsets of
  $[n]$ which contain $\{n-1,n\}$ as a subset.
\end{enumerate}

For convenience, for a set of integers $X$ and an integer $y$, we write $X+y$ to denote the set $\{x+y\,:\,x\in X\}$, and for two sets $X$ and $Y$, we write $X+Y$ for $\{x+y\,:\,x\in X,\,y\in Y\}$. 

\begin{lemma}\label{lem:recursions}
  \begin{align}
    S(n,l) &\supseteq S(n-1,l)+\binom{n-1}{l-1} &&\text{for }2\leq l\leq n-1,\label{eq:adding_isolated_vertex}\\
    S(n,l) &\supseteq S(n-1,l-1)+\binom{n-1}{l+1}&&\text{for }3\leq l\leq n-2,\label{eq:recursion_1}\\
    S(n,l) &\supseteq S(n-2,l-1)+S(n-2,l-1)+\binom{n-2}{l+1}+\binom{n-2}{l-2}&&\text{for }3\leq l\leq n-3. \label{eq:recursion_2}
  \end{align}  
\end{lemma}
\begin{proof}
  For~(\ref{eq:adding_isolated_vertex}), let $\mathcal A'\subseteq\binom{[n-1]}{[l]}\cup\binom{[n-1]}{[l+1]}$ be a maximal antichain in $B_{n-1}$. Then
  \[\mathcal A=\mathcal A'\cup\left\{A\cup\{n\}\,:\,A\in\binom{[n-1]}{l-1}\right\}\subseteq\binom{[n]}{l}\cup\binom{[n]}{l+1}\]
  is a maximal antichain in $B_n$. For~(\ref{eq:recursion_1}), let $\mathcal A'\subseteq\binom{[n-1]}{[l-1]}\cup\binom{[n-1]}{[l]}$ be a maximal antichain in $B_{n-1}$. Then
  \[\mathcal A=\left\{A\cup\{n\}\,:\,A\in\mathcal A'\right\}\cup\binom{[n-1]}{l+1}\subseteq\binom{[n]}{l}\cup\binom{[n]}{l+1}\]
  is a maximal antichain in $B_n$. For~(\ref{eq:recursion_2}), let $\mathcal A',\mathcal A''\subseteq\binom{[n-2]}{l-1}\cup\binom{[n-2]}{l}$ be maximal antichains in $B_{n-2}$. Then
  \[\mathcal A=\left\{A\cup\{n-1\}\,:\,A\in\mathcal A'\right\}\cup\left\{A\cup\{n\}\,:\,A\in\mathcal A''\right\}\cup\left\{A\cup\{n-1,n\}\,:\,A\in\binom{[n-2]}{l-2}\right\}\cup\binom{[n-2]}{l+1}\]
  is a maximal antichain in $B_n$. To see this, write $\mathcal A'=\mathcal A'_1\cup\mathcal A'_2$, $\mathcal A''=\mathcal A''_1\cup\mathcal A''_2$, $\mathcal A=\mathcal A_1\cup\mathcal A_2$ with $\mathcal A'_1,\mathcal A''_1\subseteq\binom{[n-2]}{l-1}$, $\mathcal A'_2,\mathcal  A''_2\subseteq\binom{[n-2]}{l}$, $\mathcal A_1\subseteq\binom{[n]}{l}$, and $\mathcal A_2\subseteq\binom{[n]}{l+1}$. Then
  \begin{align*}
    \mathcal A_1 &= \left\{A\cup\{n-1\}\,:\,A\in\mathcal A'_1\right\}\cup\left\{A\cup\{n\}\,:\,A\in\mathcal A''_1\right\}\cup\left\{A\cup\{n-1,n\}\,:\,A\in\binom{[n-2]}{l-2}\right\},\\
    \mathcal A_2 &= \left\{A\cup\{n-1\}\,:\,A\in\mathcal A'_2\right\}\cup\left\{A\cup\{n\}\,:\,A\in\mathcal A''_2\right\}\cup\binom{[n-2]}{l+1}.
  \end{align*}
  This implies
  \begin{align*}
    \Delta\mathcal A_2 &= \left\{A\cup\{n-1\}\,:\,A\in\Delta\mathcal A'_2\right\}\cup\left\{A\cup\{n\}\,:\,A\in\Delta\mathcal A''_2\right\}\cup\binom{[n-2]}{l},\\
    \nabla\mathcal A_1 &= \left\{A\cup\{n-1\}\,:\,A\in\nabla'\mathcal A'_1\right\}\cup\left\{A\cup\{n\}\,:\,A\in\nabla'\mathcal A''_1\right\}\cup\left\{A\cup\{n-1,n\}\,:\,A\in\binom{[n-2]}{l-1}\right\},
  \end{align*}
  where $\nabla'$ denotes the shade in $B_{n-2}$, that is, $\nabla'(A)=\{A\cup\{i\}\,:\,i\in[n-2]\setminus A\}$. Now the required equations $\mathcal A_1=\binom{[n]}{l}\setminus\Delta\mathcal A_2$ and $\mathcal A_2=\binom{[n]}{l+1}\setminus\nabla\mathcal A_1$ follow from 
  \begin{align*}
    \mathcal A'_1 &= \binom{[n-2]}{l-1}\setminus\Delta\mathcal A'_2, & \mathcal A'_2 &= \binom{[n-2]}{l}\setminus\nabla'\mathcal A'_1, & \mathcal A''_1 &= \binom{[n-2]}{l-1}\setminus\Delta\mathcal A''_2, & \mathcal A''_2 &= \binom{[n-2]}{l}\setminus\nabla'\mathcal A''_1.\qedhere
  \end{align*}
\end{proof}

\subsection{The base case}\label{subsec:k_2}
In this subsection we prove the statement of \Cref{prop:large_flat} for $l=2$.
\begin{lemma}\label{lem:base_case}
  For every $n\geq 6$, $\left[\binom{n}{2}-6,\ \binom{n}{3}-(n-3)\left\lceil\frac{n-2}{2}\right\rceil+f(n-2)\right]\subseteq S(n,2)$. 
\end{lemma}
For the lower end of the interval, a maximal antichain of size $\binom{n}{2}-6$ is $\mathcal A=\mathcal F\cup\binom{[n]}{2}\setminus\Delta\mathcal F$ where $\mathcal F=\{123,124,356,456\}$. For the upper end of the interval, a maximal antichain of size $\binom{n}{3}-(n-3)\left\lceil\frac{n-2}{2}\right\rceil+f(n-2)$ is obtained by taking as $2$-sets the edges of a forest on $n$ vertices with $\left\lceil\frac{n-2}{2}\right\rceil$ edges with exactly $f(n-2)$ pairs of adjacent edges (the existence of these forests is the contents of \Cref{lem:j^*_bound}). For example, the last forest in \Cref{fig:graphs_13_11_d} corresponds to a maximal antichain $\mathcal A\subseteq\binom{[13]}{2}\cup\binom{[13]}{3}$ with six 2-sets and size $\abs{A}=\binom{13}{3}-10\times 6+f(11)=\binom{13}{3}-50$. \Cref{lem:base_case} states that all integers in between are also obtained as sizes of flat maximal antichains. 

In~\cite{Gruettmueller2009} it was observed that the sets $S(n,2)$ can be characterized in terms of graphs with the property that every edge is contained in a triangle. We call such graphs \emph{T-graphs}. For a graph $G$, let $e(G)$ be the number of edges, and let $t(G)$ be the number of triangles. Moreover, let $\overline G$ denote the complement, that is, the graph with the same vertex set as $G$, where two vertices are adjacent in $\overline G$ if and only if they are non-adjacent in $G$.
\begin{lemma}\label{lem:graph_AC_correspondence}
  $S(n,2)=\left\{t(G)+e\left(\overline G\right)\,:\,\text{$G$ is a T-graph on $n$ vertices.}\right\}$.
\end{lemma}
\begin{proof}
  Let $G$ be a T-graph. We can take the triangles of $G$ as the $3$-sets of the antichain and the non-edges as the $2$-sets to obtain a maximal antichain of size $t(G)+e\left(\overline G\right)$. Conversely, from a maximal antichain $\mathcal A\subseteq\binom{[n]}{2}\cup\binom{[n]}{3}$, we get the required T-graph by taking $\binom{[n]}{2}\setminus\mathcal A$ as the edge set.
\end{proof}
In the following lemma, we express the size of the antichain corresponding to a T-graph $G$ in terms of $G$ and the line graph $L\left(\overline G\right)$, that is, the graph whose vertices are the edges of $\overline G$ and two of these edges are adjacent in $L\left(\overline G\right)$ if they have a vertex in common.
\begin{lemma}\label{lem:AC_size}
  Let $G$ be a T-graph. The size of the corresponding maximal flat antichain $\mathcal A$ is
  \[\abs{\mathcal A}=\binom{n}{3}-(n-3)e\left(\overline G\right)+e\left(L\left(\overline G\right)\right)-t\left(\overline G\right).\]
\end{lemma}
\begin{proof}
  This follows by the principle of inclusion and exclusion:
  \[\lvert\mathcal A\rvert=t(G)+e\left(\overline G\right)=\binom{n}{3}-e\left(\overline G\right)(n-2)+e\left(L\left(\overline G\right)\right)-t\left(\overline G\right)+e\left(\overline G\right).\qedhere\]
\end{proof}
\Cref{lem:graph_AC_correspondence,lem:AC_size} provide a correspondence between the elements of $S(n,2)$ and certain parameters of T-graphs. As a consequence, the proof of \Cref{lem:base_case} can be done in terms of T-graphs. Our strategy for proving \Cref{lem:base_case} is summarized as follows.
\begin{itemize}
\item We start with a collection of T-graphs $G_j$, $j=0,1,\dots,f(n-2)$ on vertex set $[n]$, which satisfy $e\left(\overline G_j\right)=\left\lceil(n-2)/2\right\rceil$, $t\left(\overline G_j\right)=0$ and $e\left(L\left(\overline G_j\right)\right)=j$. By \Cref{lem:AC_size}, the corresponding maximal antichains have sizes in the interval
  \[\left[\binom{n}{3}-(n-3)\left\lceil\frac{n-2}{2}\right\rceil,\,\binom{n}{3}-(n-3)\left\lceil\frac{n-2}{2}\right\rceil+f(n-2)\right].\]
\item In every $G_j$, vertex $n$ has degree $n-1$, and we delete edges $\{n,n-1\}$, $\{n,n-2\}$,\dots,$\{n,3\}$ one by one. These edge deletions preserve the property that every edge is contained in a triangle, and we can express the effect on the size of the corresponding maximal antichain: we add one 2-set (the deleted edge) and we remove all 3-sets containing the deleted edge. This gives a collection of overlapping short intervals in $S(n,2)$ whose union is a longer interval.
\item We conclude the proof of \Cref{lem:base_case} using induction on $n$ and the recursion~(\ref{eq:adding_isolated_vertex}).
\end{itemize}
For the rest of the subsection we argue in terms of T-graphs without explicitly stating the translation into sizes of maximal antichains which is provided by \Cref{lem:graph_AC_correspondence,lem:AC_size}. We start by defining the almost complete graphs $G_j$ which will serve as the starting points of the construction.
\begin{definition}\label{def:2-starter}
  An $n$-\emph{starter} is a graph $G$ on $n$ vertices whose complement $\overline G$ is a forest with $\left\lceil\frac{n-2}{2}\right\rceil$ edges. 
\end{definition}
An $n$-starter $G$ is a T-graph with triangle-free complement, and by \Cref{lem:AC_size}, the corresponding maximal antichain has size $\binom{n}{3}-(n-3)\left\lceil\frac{n-2}{2}\right\rceil+e\left(L\left(\overline G\right)\right)$. By varying $e\left(L\left(\overline G\right)\right)$ we obtain the final part of the interval in \Cref{prop:large_flat}.
\begin{example}\label{ex:matching}
  The graph $G$ which is obtained by removing the matching 
  \[E\left(\overline G\right)=\left\{\{n-1,1\},\{n-2,2\},\dots,\left\{\left\lceil\frac{n+1}{2}\right\rceil,\left\lceil\frac{n-2}{2}\right\rceil\right\}\right\}\]
  from $K_n$ is an $n$-starter with $e\left(L\left(\overline G\right)\right)=0$, and the corresponding antichain has size $\binom{n}{3}-(n-3)\left\lceil\frac{n-2}{2}\right\rceil$.
\end{example}
\begin{example}
  \Cref{fig:intersection_graph} shows the complement $\overline G$ of a starter for $n=14$ and its line graph. The corresponding antichain has size $\binom{14}{3}-11\times 6+9=\binom{14}{3}-57$.
\begin{figure}[htb]
  \begin{minipage}{.49\linewidth}
    \centering
    \begin{tikzpicture}
      \foreach \a in {1,2,...,14} {
        \node[circle,fill=black,outer sep=1pt,inner sep=1pt,label=\a*360/14:$\a$] (\a) at (\a*360/14:1.5cm) {};
      }
      \draw[thick] (13) -- (10);
      \draw[thick] (12) -- (10);
      \draw[thick] (11) -- (10);
      \draw[thick] (10) -- (1);
      \draw[thick] (9) -- (1);
      \draw[thick] (8) -- (1);
    \end{tikzpicture}
  \end{minipage}\hfill
  \begin{minipage}{.49\linewidth}
  \centering
  \begin{tikzpicture}[every node/.style={circle,fill=black,outer sep=1pt,inner sep=1pt}]
      \node (v1) at (-1,0) {};
      \node (v2) at (0,1) {};
      \node (v3) at (0,-1) {};
      \node (v4) at (1,0) {};
      \node (v5) at (3,1) {};
      \node (v6) at (3,-1) {};
      \draw[thick] (v1) -- (v2) -- (v3) -- (v4) -- (v1) -- (v3);
      \draw[thick] (v2) -- (v4) -- (v5) -- (v6) -- (v4);
    \end{tikzpicture}
  \end{minipage}
  \caption{The complement $\overline G$ of a starter $G$ for $n=14$ and the line graph $L(\overline G)$.}\label{fig:intersection_graph}
\end{figure}
\end{example}
By \Cref{ex:matching} there is always a starter $G$ with $e\left(L\left(\overline G\right)\right)=0$. In view of the comment after \Cref{def:2-starter} we would like to have a family of starters such that $e\left(L\left(\overline G\right)\right)$ varies over an interval starting at $0$. We now define the maximal length of such an interval.
\begin{definition}
  Let $i^*(n)$ be the largest integer such that for every $i\in\{0,1,\dots,i^*(n)\}$ there exists an $n$-starter $G$ with $e\left(L\left(\overline G\right)\right)=i$.
\end{definition}
For our purpose (the proof of \Cref{lem:base_case}) we only need a crude lower bound on $i^*(n)$, which is provided in the next two lemmas.
\begin{lemma}\label{lem:istar_monotonicity}
  \begin{enumerate}[(i)]
  \item For every integer $n\geq 3$, $i^*(n+1)\geq i^*(n)$ with equality if $n$ is odd.
  \item For integers $n=2k-1\geq 5$, $t\leq k-1$, and an $(n-2t)$-starter $G$, there is an $n$-starter $G'$ with $e\left(L\left(\overline{G'}\right)\right)=e\left(L\left(\overline{G}\right)\right)+\binom{t+1}{2}$.
  \end{enumerate}
\end{lemma}
\begin{proof} Part (i) is obvious. For part (ii), let $G$ be a $(2k-2t-1)$-starter. We obtain the required $(2k-1)$-starter $G'$ by adding $2t$ new vertices and joining a leaf of $G$ to $t$ of the new vertices.
\end{proof}
\begin{lemma}\label{lem:j^*_bound}
  \[i^*(n)\geq 
    \begin{cases}
      0 & \text{for }n\in\{3,4\},\\
      1 & \text{for }n\in\{5,6\},\\
      3 & \text{for }n\in\{7,8\},\\
      4 & \text{for }n\in\{9,10\},\\
      7 & \text{for }n\in\{11,12\},\\
      n-2 & \text{for odd }n\geq 13,\\
      n-3 & \text{for even }n\geq 14.
    \end{cases}
    \]
    In particular, $i^*(n)\geq f(n-2)$ for all $n\geq 3$.
\end{lemma}
\begin{proof}
  By \Cref{lem:istar_monotonicity}(i), we can focus on the case that $n$ is odd. For $n=3$, the starter $G$ with $E\left(\overline G\right)=\{\{1,2\}\}$ witnesses $i^*(3)\geq 0$. 
  \begin{description}
  \item[$n=5$] \Cref{lem:istar_monotonicity}(ii) with $t=1$ gives a $5$-starter with $e\left(L\left(\overline G\right)\right)=1$.
  \item[$n=7$] \Cref{lem:istar_monotonicity}(ii) with $t=1$ gives a starter with $e\left(L\left(\overline G\right)\right)=2$, and the complement of a $3$-star (plus 3 isolated vertices) is a $7$-starter with $e\left(\overline G\right)=3$.
  \item[$n=9$] \Cref{lem:istar_monotonicity}(ii) with $t=1$ gives a $9$-starter with $e\left(L\left(\overline G\right)\right)=4$.
  \item[$n=11$] \Cref{lem:istar_monotonicity}(ii) with $t=2$ gives $11$-starters $G_i$ with $e\left(L\left(\overline{G_i}\right)\right)=i$ for all $i\in\{3,4,5,6\}$, and together with the $11$-starter whose complement is a tree with degrees $4,2,1,1,1,1$ (plus 5 isolated vertices) we conclude $i^*(11)\geq 7$.
  \item[$n=13$] \Cref{lem:istar_monotonicity}(ii) with $t=3$ gives $13$-starters $G_i$ with $e\left(L\left(\overline{G_i}\right)\right)=i$ for $i\in\{6,7,8,9\}$, and with $t=4$ we cover $i\in\{10,11\}$. 
  \item[$n=15$] \Cref{lem:istar_monotonicity}(ii) with $t=4$ gives $15$-starters $G_i$ with $e\left(L\left(\overline{G_i}\right)\right)=i$ for $i\in\{10,11,12,13\}$
  \item[$n=17$] \Cref{lem:istar_monotonicity}(ii) with $t=4$ gives $17$-starters $G_i$ with $e\left(L\left(\overline{G_i}\right)\right)=i$ for $i\in\{10,\dots,14\}$, and with $t=5$ we cover $i\in\{15,16,17,18\}$.
  \item[$n\geq 19$] \Cref{lem:istar_monotonicity}(ii) with $t=3$ gives $n$-starters $G_i$ with $e\left(L\left(\overline{G_i}\right)\right)=i$ for $i\in\left\{6,\dots,i^*(n-6)+6\right\}$, and this concludes the proof because $i^*(n-6)\geq n-8$ by induction on $n$.  \qedhere
  \end{description}
\end{proof}
As a consequence of \Cref{lem:j^*_bound}, to prove \Cref{lem:base_case} it suffices to verify
\[\left[\binom{n}{2}-6,\,\binom{n}{3}-(n-3)\left\lceil\frac{n-2}{2}\right\rceil+i^*(n)\right]\subseteq S(n,2)\]
So far we have established the existence of starters corresponding to maximal antichains $\mathcal A\subseteq\binom{[n]}{2}\cup\binom{[n]}{3}$ with all sizes of the form
\[\binom{n}{3}-(n-3)\left\lceil\frac{n-2}{2}\right\rceil+i\qquad\text{for }i=0,1,2,\dots,i^*(n).\]
We can always label the vertices of an $n$-starter $G$ in such a way that
\begin{enumerate}[(i)]
\item $n$ is an isolated vertex in $\overline G$ (that is, $\{i,n\}\in E(G)$ for all $i\in[n-1]$),
\item for $i=1,2,\dots,\left\lceil\frac{n-2}{2}\right\rceil$, the vertex $n-i$ has degree $1$ in $\overline{G}-\{n,n-1,\dots,n-i+1\}$.
\end{enumerate}
Let us call an $n$-starter satisfying (i) and (ii) \emph{properly labeled}. 
The second condition (ii) says that the edge set of $\overline{G}$ has the form $E\left(\overline{G}\right)=\{\{n-l,b_l\}\,:\,l=1,\dots,\left\lceil\frac{n-2}{2}\right\rceil\}$ with $b_l<n-l$ for every $l$. In this setup, we will now delete the edges $\{n,n-1\},\dots\{n,3\}$ from $G$ one by one, and this shifts the sizes of the corresponding maximal antichains, as illustrated in the following example.
\begin{example}
  For $n=14$, the starters for $i=0,1,\dots,i^*(14)=11$ correspond to maximal antichains with sizes in the interval $[\left[\binom{14}{3}-66,\binom{14}{3}-55\right]$. In \Cref{fig:shift}, it is illustrated how deleting the edges $\{14,13\}$, $\{14,12\}$,\dots,$\{14,3\}$ shifts the size of the corresponding antichain.
  \begin{figure}[htb]
  \begin{minipage}{.32\linewidth}
    \centering
  \begin{tikzpicture}[scale=.9]
      \foreach \a in {1,2,...,14} {
        \node[circle,fill=black,outer sep=1pt,inner sep=1pt,label=\a*360/14:{\small $\a$}] (\a) at (\a*360/14:2cm) {};
      }
      \draw[thick] (13) -- (10);
      \draw[thick] (12) -- (10);
      \draw[thick] (11) -- (10);
      \draw[thick] (10) -- (1);
      \draw[thick] (9) -- (1);
      \draw[thick] (8) -- (1);
      \node at (0,-3) {$\binom{14}{3}-66+9=\binom{14}{3}-57$};
  \end{tikzpicture}  
\end{minipage}\hfill
\begin{minipage}{.32\linewidth}
    \centering
  \begin{tikzpicture}[scale=.9]
      \foreach \a in {1,2,...,14} {
        \node[circle,fill=black,outer sep=1pt,inner sep=1pt,label=\a*360/14:{\small $\a$}] (\a) at (\a*360/14:2cm) {};
      }
      \draw[thick] (13) -- (10);
      \draw[thick] (12) -- (10);
      \draw[thick] (11) -- (10);
      \draw[thick] (10) -- (1);
      \draw[thick] (9) -- (1);
      \draw[thick] (8) -- (1);
      \draw[thick] (14) -- (13);
      \node at (0,-3) {$\binom{14}{3}-67$};      
    \end{tikzpicture}
  \end{minipage}\hfill
  \begin{minipage}{.32\linewidth}
    \centering
  \begin{tikzpicture}[scale=.9]
      \foreach \a in {1,2,...,14} {
        \node[circle,fill=black,outer sep=1pt,inner sep=1pt,label=\a*360/14:{\small $\a$}] (\a) at (\a*360/14:2cm) {};
      }
      \draw[thick] (13) -- (10);
      \draw[thick] (12) -- (10);
      \draw[thick] (11) -- (10);
      \draw[thick] (10) -- (1);
      \draw[thick] (9) -- (1);
      \draw[thick] (8) -- (1);
      \draw[thick] (14) -- (13);
      \draw[thick] (14) -- (12);
      \node at (0,-3) {$\binom{14}{3}-76$};      
    \end{tikzpicture}    
  \end{minipage}

  \bigskip

  \begin{minipage}{.32\linewidth}
    \centering
  \begin{tikzpicture}[scale=.9]
      \foreach \a in {1,2,...,14} {
        \node[circle,fill=black,outer sep=1pt,inner sep=1pt,label=\a*360/14:{\small $\a$}] (\a) at (\a*360/14:2cm) {};
      }
      \draw[thick] (13) -- (10);
      \draw[thick] (12) -- (10);
      \draw[thick] (11) -- (10);
      \draw[thick] (10) -- (1);
      \draw[thick] (9) -- (1);
      \draw[thick] (8) -- (1);
      \draw[thick] (14) -- (13);
      \draw[thick] (14) -- (12);
      \draw[thick] (14) -- (11);
      \node at (0,-3) {$\binom{14}{3}-84$};      
    \end{tikzpicture}    
  \end{minipage}$\dots$
    \begin{minipage}{.32\linewidth}
    \centering
  \begin{tikzpicture}[scale=.9]
      \foreach \a in {1,2,...,14} {
        \node[circle,fill=black,outer sep=1pt,inner sep=1pt,label=\a*360/14:{\small $\a$}] (\a) at (\a*360/14:2cm) {};
      }
      \draw[thick] (13) -- (10);
      \draw[thick] (12) -- (10);
      \draw[thick] (11) -- (10);
      \draw[thick] (10) -- (1);
      \draw[thick] (9) -- (1);
      \draw[thick] (8) -- (1);
      \draw[thick] (14) -- (13);
      \draw[thick] (14) -- (12);
      \draw[thick] (14) -- (11);
      \draw[thick] (14) -- (10);
      \draw[thick] (14) -- (9);
      \draw[thick] (14) -- (8);
      \draw[thick] (14) -- (7);
      \draw[thick] (14) -- (6);
      \draw[thick] (14) -- (5);
      \draw[thick] (14) -- (4);
      \node at (0,-3) {$\binom{14}{3}-116$};      
    \end{tikzpicture}    
  \end{minipage}\hfill
    \begin{minipage}{.32\linewidth}
    \centering
  \begin{tikzpicture}[scale=.9]
      \foreach \a in {1,2,...,14} {
        \node[circle,fill=black,outer sep=1pt,inner sep=1pt,label=\a*360/14:{\small $\a$}] (\a) at (\a*360/14:2cm) {};
      }
      \draw[thick] (13) -- (10);
      \draw[thick] (12) -- (10);
      \draw[thick] (11) -- (10);
      \draw[thick] (10) -- (1);
      \draw[thick] (9) -- (1);
      \draw[thick] (8) -- (1);
      \draw[thick] (14) -- (13);
      \draw[thick] (14) -- (12);
      \draw[thick] (14) -- (11);
      \draw[thick] (14) -- (10);
      \draw[thick] (14) -- (9);
      \draw[thick] (14) -- (8);
      \draw[thick] (14) -- (7);
      \draw[thick] (14) -- (6);
      \draw[thick] (14) -- (5);
      \draw[thick] (14) -- (4);
      \draw[thick] (14) -- (3);
      \node at (0,-3) {$\binom{14}{3}-117$};      
    \end{tikzpicture}    
    \end{minipage}
    \caption{Deleting the edges $\{n,n-1\},\{n,n-2\},\dots,\{n,3\}$ from a starter leads to graphs corresponding to maximal antichains. The number under each graph is the size of the corresponding maximal antichain.}\label{fig:shift}
\end{figure}
Doing this for all starters, we obtain
\begin{multline*}
  S(14,2)\supseteq\binom{14}{3}-\big([55,66]\cup[65,76]\cup[74,85]\cup[82,93]\cup[89,100]\cup[95,106]\cup[100,111]\cup[105,116]\\
  \cup[109,120]\cup[112,123]\cup[114,125]\cup[115,126]\big)=\binom{14}{3}-[55,126].
\end{multline*}
This interval is the first part of the interval in the statement of \Cref{lem:base_case} for $n=14$. The missing sizes between $\binom{14}{2}-5=86$ and $\binom{14}{3}-126=238$ will come from induction on $n$.
\end{example}
In the next lemma we state precisely how deleting the edges incident with vertex $n$ from a starter affects the size of the corresponding maximal antichain. For a graph $G$ on $n$ vertices, let
\[\phi(G)=\binom{n}{3}-(n-3)e\left(\overline{G}\right)+e\left(L\left(\overline{G}\right)\right)-t\left(\overline{G}\right).\]
By \Cref{lem:AC_size}, if $G$ is a $T$-graph then $\phi(G)$ is the size of the corresponding maximal antichain $\mathcal A$.
\begin{lemma}\label{lem:move_sequence}
  Let $n\geq 6$, and let $G_0$ be a properly labeled $n$-starter. Furthermore, for $l=1,2,\dots,n-3$, let $G_l$ be the graph obtained from $G_0$ by deleting the edges $\{n,n-1\}$, $\{n,n-2\}$,\dots, $\{n,n-l\}$. Then, for every $l\in[n-1]$, $\phi\left(G_l\right)=\phi\left(G_{l-1}\right)-\alpha_l$, where
  \[\alpha_l=
    \begin{cases}
      n-l-3 &\text{for }l=1,2,\dots,\left\lceil\frac{n-2}{2}\right\rceil,\\
      n-l-2 &\text{for }l=\left\lceil\frac{n-2}{2}\right\rceil+1,\dots,n-1.
    \end{cases}
  \]
  Moreover, the graphs $G_1,\dots,G_{n-3}$ are T-graphs.
\end{lemma}
\begin{proof}
  In each step one edge is deleted, hence
  \begin{equation}\label{eq:delete_edge}
    e\left(\overline{G_l}\right)=e\left(\overline{G_{l-1}}\right)+1.
  \end{equation}
  The new edge $\{n,n-l\}$ becomes a new vertex in the line graph. This new vertex is adjacent to
  \begin{itemize}
  \item all edges $\{n,n-l'\}$, $l'=1,\dots,l-1$, 
  \item all $\{n-l,n-j\}\in E\left(\overline{G_0}\right)$ with $j\in\{1,\dots,l-1\}$, and
  \item the edge $\{n-l,b_l\}$ if $l\leq\lceil(n-2)/2\rceil$. 
  \end{itemize}
  The second group of additional edges in the line digraph is compensated by additional triangles in $\overline{G_l}$. More precisely, 
  \begin{align}
    e\left(L\left(\overline{G_l}\right)\right) &= e\left(L\left(\overline{G_{l-1}}\right)\right)+(l-1)+\left\lvert\left\{j\,:\,\{n-l,n-j\}\in E\left(\overline{G_0}\right),\,1\leq j\leq l-1\right\}\right\rvert\nonumber\\
    &\qquad\qquad+
    \begin{cases}
    1 &\text{if }l\leq\lceil(n-2)/2\rceil\\
    0 &\text{if }l>\lceil(n-2)/2\rceil
    \end{cases}\label{eq:new_edges}\\
    t\left(\overline{G_l}\right) &= t\left(\overline{G_{l-1}}\right)+\left\lvert\left\{j\,:\,\{n-l,n-j\}\in E\left(\overline{G_0}\right),\,1\leq j\leq l-1\right\}\right\rvert.\label{eq:new_triangles}
  \end{align}
  Combining~(\ref{eq:delete_edge}),~(\ref{eq:new_edges}) and~(\ref{eq:new_triangles}), we get
  \[\phi\left(G_l\right)=\phi\left(G_{l-1}\right)-(n-3)+(l-1)+
    \begin{cases}
      1 &\text{if }l\leq\lceil(n-2)/2\rceil,\\
      0 &\text{if }l>\lceil(n-2)/2\rceil.
    \end{cases}\]
  To see that the graphs $G_l$, $l=1,\dots,n-3$, are T-graphs we check that every edge $\{p,q\}$ is contained in a triangle. For $p,q<n$, the edges of $\overline{G_0}$ remove at most $\lceil(n-2)/2\rceil$ of the $n-2$ triangles $\{p,q,r\}$, and deleting the edges $\{n,l\}$ can remove only one additional triangle, namely the triangle $\{p,q,n\}$. From $\lceil(n-2)/2\rceil+1<n-2$ it follows that there is always a triangle $\{p,q,r\}$ left. For edges $\{p,n\}$ with $2\leq p\leq n/2$, there is the triangle $\{1,p,n\}$, for the edge $\{1,n\}$ there is the triangle $\{1,2,n\}$, and for an edge $\{p,n\}$ with $p=n-i$, $i\in\{1,\dots,\lceil(n-2)/2\rceil\}$, we can take any triangle $\{j,p,n\}$ with $j\in\{1,\dots,p-1\}\setminus\{b_i\}$.
\end{proof}
In the next lemma, we state the interval of maximal antichain sizes obtained by the above construction.
\begin{lemma}\label{lem:top_row}
  For every integer $n\geq 6$,  
  \[S(n,2)\supseteq\left[\binom{n}{3}-\binom{n-2}{2}-(n-4)\left\lceil\frac{n-2}{2}\right\rceil,\,\binom{n}{3}-(n-3)\left\lceil\frac{n-2}{2}\right\rceil+i^*(n)\right].\]  
\end{lemma}
\begin{proof}
  For $n=6$, we obtain the sizes in the interval $[10,15]$ from the T-graphs whose complements are shown in \Cref{fig:n_6}.
  \begin{figure}[htb]
    \begin{minipage}{.32\linewidth}
      \centering
    \begin{tikzpicture}
      \foreach \a in {1,2,...,6} {
        \node[circle,fill=black,outer sep=1pt,inner sep=1pt,label=\a*360/6:$\a$] (\a) at (\a*360/6:1cm) {};
      }
      \draw[thick] (5) -- (1);
      \draw[thick] (4) -- (2);
      \node at (0,-1.7) {{\small $12+2=14$}};
    \end{tikzpicture}
  \end{minipage}\hfill
  \begin{minipage}{.32\linewidth}
      \centering
  \begin{tikzpicture}
      \foreach \a in {1,2,...,6} {
        \node[circle,fill=black,outer sep=1pt,inner sep=1pt,label=\a*360/6:$\a$] (\a) at (\a*360/6:1cm) {};
      }
      \draw[thick] (5) -- (1);
      \draw[thick] (4) -- (1);
      \node at (0,-1.7) {{\small $13+2=15$}};
    \end{tikzpicture}
    \end{minipage}\hfill
  \begin{minipage}{.32\linewidth}
      \centering
  \begin{tikzpicture}
      \foreach \a in {1,2,...,6} {
        \node[circle,fill=black,outer sep=1pt,inner sep=1pt,label=\a*360/6:$\a$] (\a) at (\a*360/6:1cm) {};
      }
      \draw[thick] (5) -- (1);
      \draw[thick] (4) -- (2);
      \draw[thick] (6) -- (5);
      \node at (0,-1.7) {{\small $9+3=12$}};
    \end{tikzpicture}
  \end{minipage}

  \bigskip
  
  \begin{minipage}{.32\linewidth}
      \centering
    \begin{tikzpicture}
      \foreach \a in {1,2,...,6} {
        \node[circle,fill=black,outer sep=1pt,inner sep=1pt,label=\a*360/6:$\a$] (\a) at (\a*360/6:1cm) {};
      }
      \draw[thick] (5) -- (1);
      \draw[thick] (4) -- (1);
      \draw[thick] (6) -- (5);
      \node at (0,-1.7) {{\small $10+3=13$}};
    \end{tikzpicture}
  \end{minipage}\hfill
  \begin{minipage}{.32\linewidth}
      \centering
  \begin{tikzpicture}
      \foreach \a in {1,2,...,6} {
        \node[circle,fill=black,outer sep=1pt,inner sep=1pt,label=\a*360/6:$\a$] (\a) at (\a*360/6:1cm) {};
      }
      \draw[thick] (5) -- (1);
      \draw[thick] (4) -- (2);
      \draw[thick] (6) -- (5);
      \draw[thick] (6) -- (4);
      \draw[thick] (6) -- (3);
      \node at (0,-1.7) {{\small $5+5=10$}};
    \end{tikzpicture}
    \end{minipage}\hfill
  \begin{minipage}{.32\linewidth}
      \centering
  \begin{tikzpicture}
      \foreach \a in {1,2,...,6} {
        \node[circle,fill=black,outer sep=1pt,inner sep=1pt,label=\a*360/6:$\a$] (\a) at (\a*360/6:1cm) {};
      }
      \draw[thick] (5) -- (1);
      \draw[thick] (4) -- (1);
      \draw[thick] (6) -- (5);
      \draw[thick] (6) -- (4);
      \draw[thick] (6) -- (3);
      \node at (0,-1.7) {{\small $6+5=11$}};
    \end{tikzpicture}
    \end{minipage}
      \caption{Complements of the graphs corresponding to maximal antichains with sizes in $[10,15]$. The line below each graph has the format ``$\# 3\text{-sets}+\# 2\text{-sets} =\text{size of the antichain}$''.}\label{fig:n_6}
    \end{figure}
  For $n\geq 7$, let $m$ be any number in the interval on the right hand side of the claimed inclusion, and let $m_0$ be the smallest size of an antichain coming from a starter, that is,
  \[m_0=\binom{n}{3}-(n-3)\left\lceil\frac{n-2}{2}\right\rceil.\]
  By definition of $i^*(n)$, there is nothing to do if $m\in[m_0,m_0+i^*(n)]$. We define the sequence $\alpha_1,\dots,\alpha_{n-1}$ as in the proof of \Cref{lem:move_sequence}. Using \Cref{lem:move_sequence}, we can obtain a maximal antichain of size
  \[m_0-\left(\alpha_1+\alpha_2+\dots+\alpha_{n-3}\right) = m_0-\binom{n-2}{2}+\left\lceil\frac{n-2}{2}\right\rceil=\binom{n}{3}-\binom{n-2}{2}-(n-4)\left\lceil\frac{n-2}{2}\right\rceil.\]
  Let $t\in\{0,1,\dots,n-3\}$ be the smallest index with $m_0-\left(\alpha_1+\alpha_2+\dots+\alpha_{t}\right)\leq m$, and set $i=m-m_0+\left(\alpha_1+\alpha_2+\dots+\alpha_{t}\right)$. Then $i\leq\alpha_1-1=n-5$ with equality only if $t=1$ and $m=m_0-1$. By \Cref{lem:j^*_bound}, this leaves two cases.
  \begin{description}
  \item[Case 1] $i\leq i^*(n)$. Let $G$ be a starter with $t(G)+e\left(\overline G\right)=m_0+i$. By \Cref{lem:move_sequence}, the graph $G'$ with 
    \[ E\left(\overline{G'}\right) = E\left(\overline G\right)\cup\left\{\{n,n-1\},\,\{n,n-2\},\dots,\,\{n,n-t\}\right\},\]
    corresponds to a maximal antichain of size $m$.
\item[Case 2] $n=10$ and $m=m_0-1=91$. In this case, we conclude by pointing to the T-graph $G$ with 
  \[E\left(\overline G\right) = \left\{\{10,9\},\,\{1,2\},\,\{1,3\},\,\{1,4\},\,\{1,5\}\right\},\]
 which satisfies $t(G)+e\left(\overline G\right)=86+5=91$.\qedhere
  \end{description}
\end{proof}
Finally, we combine \Cref{lem:top_row} with the recursion~(\ref{eq:adding_isolated_vertex}) to prove \Cref{lem:base_case}.
\begin{proof}[Proof of \Cref{lem:base_case}.]
  For the base case $n=6$, \Cref{lem:top_row} gives the interval $[10,15]$, and size 9 is obtained by the flat maximal antichain with 3-sets $\{1,2,5\}$, $\{1,2,6\}$, $\{3,4,5\}$, $\{3,4,6\}$. Now assume $n\geq 7$ and
  \[\left[\binom{n-1}{2}-6,\,\binom{n-1}{3}-(n-4)\left\lceil\frac{n-3}{2}\right\rceil+i^*(n-1)\right]\subseteq S(n-1,2).\]
  Then~(\ref{eq:adding_isolated_vertex}) implies
  \begin{multline*}
    S(n,2)\supseteq\left[\binom{n-1}{2}-6,\,\binom{n-1}{3}-(n-4)\left\lceil\frac{n-3}{2}\right\rceil+i^*(n-1)\right]+(n-1)\\
    = \left[\binom{n}{2}-6,\,\binom{n-1}{3}-(n-4)\left\lceil\frac{n-3}{2}\right\rceil+i^*(n-1)+(n-1)\right],   
  \end{multline*}
  while \Cref{lem:top_row} ensures that
  \[S(n,2)\supseteq\left[\binom{n}{3}-\binom{n-2}{2}-(n-4)\left\lceil\frac{n-2}{2}\right\rceil,\,\binom{n}{3}-(n-3)\left\lceil\frac{n-2}{2}\right\rceil+i^*(n)\right].\]
  To conclude the proof it is sufficient to verify 
  \[\binom{n-1}{3}-(n-4)\left\lceil\frac{n-3}{2}\right\rceil+i^*(n-1)+(n-1)\geq\binom{n}{3}-\binom{n-2}{2}-(n-4)\left\lceil\frac{n-2}{2}\right\rceil-1.\]
  This simplifies to
  \[i^*(n-1)\geq (n-4)\left(\left\lceil\frac{n-3}{2}\right\rceil-\left\lceil\frac{n-2}{2}\right\rceil\right)-2,\]
  which is obvious because the right hand side is negative.
\end{proof}

\subsection{Intervals in \texorpdfstring{$S(n,l)$}{}}\label{subsec:k_geq_3}
In this subsection we prove \Cref{prop:large_flat} by induction on $l$ and $n$. The structure of the argument can be summarized as follows.
\begin{itemize}
\item The base case for the induction on $l$, that is, the case $l=2$, is \Cref{lem:base_case} and has been proved in \Cref{subsec:k_2}.
\item For $l\geq 3$, the base case for the induction on $n$ is $n=2l+2$. By induction on $l$, we have intervals $I_1\subseteq S(2l+1,l-1)$ and
  $I_2\subseteq S(2l,l-1)$. By~(\ref{eq:recursion_1}), $I_1+\binom{2l+1}{l+1}\subseteq S(n,l)$, and by~(\ref{eq:recursion_1}) and~(\ref{eq:adding_isolated_vertex}),
  $I_2+\binom{2l}{l+1}+\binom{2l+1}{l-1}\subseteq S(n,l)$, and we verify that the union of these two intervals is the required subset of $S(n,l)$.
\item For $l\geq 3$ and $n\geq 2l+3$, we have three ingredients from the induction, and apply (\ref{eq:adding_isolated_vertex}), (\ref{eq:recursion_1}) and~(\ref{eq:recursion_2}):
  \begin{itemize}
  \item induction on $n$ gives us an interval $I_1\subseteq S(n-1,l)$ and then $I_1+\binom{n-1}{l-1}\subseteq S(n,l)$ by~(\ref{eq:adding_isolated_vertex}),
  \item induction on $l$ gives us an interval $I_2\subseteq S(n-1,l-1)$ and then $I_2+\binom{n-1}{l+1}\subseteq S(n,l)$ by~(\ref{eq:recursion_1}), 
  \item induction on $l$ gives us an interval $I_3\subseteq S(n-2,l-1)$ and then $I_3+I_3+\binom{n-2}{l+1}+\binom{n-2}{l-2}\subseteq S(n,l)$ by~(\ref{eq:recursion_2}).    
  \end{itemize}
  We conclude the proof of \Cref{prop:large_flat} by verifying that the union of these three intervals is the required subset of $S(n,l)$.
\end{itemize}
Recall that $C_l$ denotes the sum of the first $l$ Catalan numbers: 
$C_l=\sum_{i=1}^l\frac{1}{i+1}\binom{2i}{i}$.

\begin{proof}[Proof of Proposition~\ref{prop:large_flat}]
We proceed by induction on $l$ and $n$. The base cases ($l=2$ and $n\geq 6$) is
\Cref{lem:base_case}. For $l\geq 3$ we start with $n=2l+2$. By induction on $l$,
\[S(2l+1,l-1)\supseteq\left[\binom{2l+1}{l-1}-3-C_{l-1},\ \binom{2l+1}{l}-(l+1)\left\lceil\frac{l+2}{2}\right\rceil+f(l+2)\right],\]
and with~(\ref{eq:recursion_1}),
\begin{equation}\label{eq:first_half}
  S(2l+2,l)\supseteq\left[\binom{2l+2}{l}-3-C_{l-1},\ \binom{2l+2}{l+1}-(l+1)\left\lceil\frac{l+2}{2}\right\rceil+f(l+2)\right].
\end{equation}
Similarly,
\[S(2l,l-1)\supseteq\left[\binom{2l}{l-1}-3-C_{l-1},\ \binom{2l}{l}-l\left\lceil\frac{l+1}{2}\right\rceil+f(l+1)\right],\]
and then, using~(\ref{eq:recursion_1}) and~(\ref{eq:adding_isolated_vertex}),
\begin{multline}\label{eq:second_half}
  S(2l+2,l)\supseteq\left[\binom{2l}{l-1}-3-C_{l-1},\ \binom{2l}{l}-l\left\lceil\frac{l+1}{2}\right\rceil+f(l+1)\right]+\binom{2l}{l+1}+\binom{2l+1}{l-1}\\
  =\left[\binom{2l+2}{l}-3-C_l,\ \binom{2l+2}{l}-l\left\lceil\frac{l+1}{2}\right\rceil+f(l+1)\right].
\end{multline}
Here the left end of the interval comes from 
\[\binom{2l}{l-1}+\binom{2l}{l+1}+\binom{2l+1}{l-1}=\binom{2l+2}{l}-\frac{1}{l+1}\binom{2l}{l}.\]
Using the inequality $l\left\lceil\frac{l+1}{2}\right\rceil-f(l+1)\leq 4+C_{l-1}$, we obtain
from~(\ref{eq:first_half}) and (\ref{eq:second_half}) that
\[S(2l+2,l)\supseteq\left[\binom{2l+2}{l}-3-C_l,\
    \binom{2l+2}{l+1}-(l+1)\left\lceil\frac{l+2}{2}\right\rceil+f(l+2)\right],\] 
which is the required interval for $n=2l+2$. For $n\geq 2l+3$, we have by induction that $I_1\subseteq S(n-1,l)$, $I_2\subseteq S(n-1,l-1)$ and $I_3\subseteq S(n-2,l-1)$, where
  \begin{align*}
    I_1 &= \left[\binom{n-1}{l}-3-C_l,\ \binom{n-1}{l+1}-(n-l-2)\left\lceil\frac{n-l-1}{2}\right\rceil+f(n-l-1)\right],\\
    I_2 &= \left[\binom{n-1}{l-1}-3-C_{l-1},\ \binom{n-1}{l}-(n-l-1)\left\lceil\frac{n-l}{2}\right\rceil+f(n-l)\right],\\
    I_3 &= \left[\binom{n-2}{l-1}-3-C_{l-1},\ \binom{n-2}{l}-(n-l-2)\left\lceil\frac{n-l-1}{2}\right\rceil+f(n-l-1)\right].
  \end{align*}
  With~(\ref{eq:adding_isolated_vertex}) and~(\ref{eq:recursion_1}) we obtain from $I_1\subseteq S(n-1,l)$ and $I_2\subseteq S(n-1,l-1)$ that
  \begin{align*}
    S(n,l) &\supseteq \left[\binom{n}{l}-3-C_l,\ \binom{n-1}{l+1}+\binom{n-1}{l-1}-(n-l-2)\left\lceil\frac{n-l-1}{2}\right\rceil+f(n-l-1)\right],\\
    S(n,l) &\supseteq \left[\binom{n-1}{l+1}+\binom{n-1}{l-1}-3-C_{l-1},\ \binom{n}{l+1}-(n-l-1)\left\lceil\frac{n-l}{2}\right\rceil+f(n-l)\right].
  \end{align*}
  From~(\ref{eq:recursion_2}) and $I_3\subseteq S(n-2,l-1)$, we have
  \begin{multline*}
    S(n,l) \supseteq \left[\binom{n-2}{l-1}-3-C_{l-1},\ \binom{n-2}{l}-(n-l-2)\left\lceil\frac{n-l-1}{2}\right\rceil+f(n-l-1)\right]\\ 
    +\binom{n-2}{l}-(n-l-2)\left\lceil\frac{n-l-1}{2}\right\rceil+f(n-l-1)+\binom{n-2}{l+1}+\binom{n-2}{l-2}\\
    \supseteq\left[\binom{n-1}{l+1}+\binom{n-1}{l-1}-(n-l-2)\left\lceil\frac{n-l-1}{2}\right\rceil+f(n-l-1),\right.\\
    \left.\binom{n-1}{l+1}+\binom{n-2}{l}+\binom{n-2}{l-2}-(n-l-2)(n-l)+2f(n-l-1)\right],
  \end{multline*}
  so the result follows once we verify that
  \[\binom{n-1}{l+1}+\binom{n-2}{l}+\binom{n-2}{l-2}-(n-l-2)(n-l)+2f(n-l-1)\geq\binom{n-1}{l+1}+\binom{n-1}{l-1}-4-C_{l-1}.\]
  This inequality simplifies to
  \[\binom{n-2}{l}-\binom{n-2}{l-1}-(n-l-2)(n-l)+2f(n-l-1)+4+C_{l-1}\geq 0.\]
  For $(l,n)=(3,9)$ and $(l,n)=(3,10)$, the left-hand side is equal to 3 and 6, respectively, and for all other pairs $(l,n)$ with $l\geq 3$ and $n\geq 2l+3$, the inequality is a consequence of the following lemma.
\end{proof}
\begin{lemma}\label{lem:aux_inequality}
  For integers $l$ and $n$ with $l=3$ and $n\geq 11$, or $l\geq 4$ and $n\geq 2l+3$,
  \[\binom{n-2}{l}-\binom{n-2}{l-1}\geq(n-l-2)(n-l).\]
\end{lemma}
\begin{proof}
  The left-hand side is
  \[\binom{n-2}{l}-\binom{n-2}{l-1}=\binom{n-2}{l}-\frac{l}{n-l-1}\binom{n-2}{l}=\frac{n-2l-1}{n-l-1}\binom{n-2}{l},\]
  and after cancelling common factors the claimed inequality becomes 
  \[\frac{n-2l-1}{n-l-2}\prod_{k=n-l+1}^{n-2}k\geq l!.\] 
  For fixed $l$, the left-hand side is increasing in $n$, so it is sufficient to consider the smallest relevant $n$, that is, $n=11$ if $l=3$ and $n=2l+3$ if $l\geq 4$. For $l=3$, $\frac{4}{6}\times 9=6=3!$. For $l\geq 4$ the claim is $\frac{2}{l+1}\prod_{k=l+4}^{2l+1}k\geq l!$, and we proceed by induction on $l$. The base case $l=4$:  $\frac{2}{5}\times 8\times 9=\frac{144}{5}>4!$, and for $l\geq 5$, the induction step is
  \[\frac{2}{l+1}\prod_{k=l+4}^{2l+1}k=\frac{l(2l)(2l+1)}{(l+1)(l+3)}\left(\frac{2}{l}\prod_{k=l+3}^{2l-1}k\right)\geq\frac{l(2l)(2l+1)}{(l+1)(l+3)}(l-1)!>l!.\qedhere\]  
\end{proof}


\section{Proofs of \texorpdfstring{\Cref{prop:flat}}{} and \texorpdfstring{\Cref{thm:main_result}}{}}\label{sec:proof_main_result}
\subsection{Proof of Proposition~\ref{prop:flat}}\label{subsec:thm_flat}
In this subsection we construct small maximal antichains in $\binom{[n]}{l}\cup\binom{[n]}{l+1}$ in order to extend the interval of sizes given by \Cref{prop:large_flat} to the left. This generalizes the construction for the minimum of $S(n,2)$ from \cite{Gruettmueller2009}. 
The corresponding maximal antichain of minimum size $\binom{n}{2}-\left\lfloor\frac{(n+1)^2}{8}\right\rfloor$ is obtained by partitioning the ground set into two parts $A$ and $B$ of almost equal size with $\abs{A}$ even, say $\abs{A}=\left\lfloor\frac{n}{2}\right\rfloor$ if $n\equiv 0\text{ or }1\pmod 4$ and $\abs{A}=\left\lfloor\frac{n+2}{2}\right\rfloor$ if $n\equiv 2\text{ or }3\pmod 4$. With $M$ being a perfect matching on $A$ the minimum size of a maximal antichain in $\binom{[n]}{2}\cup\binom{[n]}{3}$ is obtained by taking the $3$-sets $xyz$ with $xy\in M$ and $z\in B$. We observe that, essentially by dropping the 3-sets one by one, we obtain flat maximal antichains of all sizes between the minimum $\binom{n}{2}-\left\lfloor\frac{(n+1)^2}{8}\right\rfloor$ and $\binom{n}{2}-2$ (coming from the flat maximal antichain which consists of the 3-set $\{1,2,3\}$ and all $2$-sets except the subsets of $\{1,2,3\}$). 
Indeed, fixing an edge $e$ in $M$, the size of the corresponding maximal antichain increases by 1 as we drop one 3-set containing $e$, as long as this 3-set is not the last one containing $e$. Dropping this last 3-set $D$ increases the size by 2. If $e$ is not the last remaining element of $M$, the missing size can be obtained by leaving $D$ in the family and dropping a 3-set containing another edge $e'\in M$.

For arbitrary $l$,  we generalize this construction as follows. Fix an integer $t$ and a family $\mathcal F\subseteq\binom{[t]}{l}$ of pairwise shadow-disjoint $l$-sets, that is, such that $\abs{A\cap B}\leq l-2$ for any distinct members $A,B\in\mathcal F$. Additionally, for each member $A\in\mathcal F$ we specify a non-empty set $X(A)\subseteq\{t+1,\dots,n\}$, and we set
\begin{align}
  \mathcal A_{l+1} &= \left\{A\cup\{i\}\,:\,A\in\mathcal F,\,i\in X(A)\right\}, & \mathcal                                                                         A_l=\binom{[n]}{l}\setminus\Delta\mathcal A_{l+1}.\label{eq:small_flat_construction}
\end{align}
Then $\mathcal A=\mathcal A_{l+1}\cup\mathcal A_{l}$ is an antichain of size
\[\abs{\mathcal A}=\sum_{A\in\mathcal F}\abs{X(A)}+\binom{n}{l}-\abs{\mathcal F}-\sum_{A\in\mathcal
    F}l\abs{X(A)}=\binom{n}{l}-\abs{\mathcal F}-(l-1)\sum_{A\in\mathcal F}\abs{X(A)}.\]
Moreover, $\mathcal{A}$ is maximal. Clearly, one cannot add an $l$-set to $\mathcal{A}$ without destroying the antichain property. 
Assume there was an $S\in\binom{[n]}{l+1}\setminus\mathcal{A}$ with $\Delta S\subseteq \Delta\mathcal{A}_{l+1}$. 
By $$\Delta\mathcal{A}_{l+1}=\mathcal{F}\cup\left\{A'\cup\{i\}:A'\in\Delta A,~A\in\mathcal{F},~i\in X(A)\right\},$$
we have $|S\cap\{t+1,\dots,n\}|\le 1$. If $S$ and $\{t+1,\dots,n\}$ are disjoint, then $\Delta S\subseteq\mathcal{F}$, a contradiction to the choice of $\mathcal{F}$.
Consequently, $S=T\cup\{x\}$ with $T\subseteq [t]$ and $x\in [n]\setminus [t]$. By $T\in\Delta S$ and $S\notin\mathcal{A}_{l+1}$, we obtain $T\in\mathcal{F}$ and $x\notin X(T)$. Let $y\in T$. As $T\setminus\{y\}\cup\{x\}\in\Delta S$, it follows that $T\setminus\{y\}\in\Delta A$ for some $A\in\mathcal{F}$ with $x\in X(A)$. Now $T$ and $A$ are distinct elements of $\mathcal{F}$ which are not shadow-disjoint. Again, this contradicts the choice of $\mathcal{F}$.
\begin{example}\label{ex:coding_construction}
    For $n=9$, $l=3$ and $t=6$, we can take for $\mathcal F$ any subfamily of $\{123,145,246,356\}$. For $\mathcal F=\{123\}$ we obtain the maximal antichains with collections of 4-sets
    \begin{itemize}
        \item $\mathcal A_4=\{1237\}$ (size $\binom{9}{3}-3=81$)
        \item $\mathcal A_4=\{1237,1238\}$ (size $\binom{9}{3}-5=79$)
        \item $\mathcal A_4=\{1237,1238,1239\}$ (size $\binom{9}{3}-7=77$)
    \end{itemize}
    For $\mathcal F=\{123,145\}$ we obtain the maximal antichains with collections of 4-sets
    \begin{itemize}
        \item $\mathcal A_4=\{1237,1457\}$ (size $\binom{9}{3}-6=78$)
        \item $\mathcal A_4=\{1237,1457,1238\}$ (size $\binom{9}{3}-8=76$)
        \item $\mathcal A_4=\{1237,1457,1238,1458\}$ (size $\binom{9}{3}-10=74$)
        \item $\mathcal A_4=\{1237,1457,1238,1458,1239\}$ (size $\binom{9}{3}-12=72$)
        \item $\mathcal A_4=\{1237,1457,1238,1458,1239,1459\}$ (size $\binom{9}{3}-14=70$)
    \end{itemize}
    For $\mathcal F=\{123,145,246\}$ we obtain the maximal antichains with collections of 4-sets
    \begin{itemize}
        \item $\mathcal A_4=\{1237,1457,2467\}$ (size $\binom{9}{3}-9=75$)
        \item $\mathcal A_4=\{1237,1457,2467,1238\}$ (size $\binom{9}{3}-11=73$)
        \item \dots
        \item $\mathcal A_4=\{1237,1457,2467,1238,1458,2468,1239,1459,2469\}$ (size $\binom{9}{3}-21=63$)
    \end{itemize}
    For $\mathcal F=\{123,145,246,357\}$ we obtain the maximal antichains with collections of 4-sets
    \begin{itemize}
        \item $\mathcal A_4=\{1237,1457,2467,3567\}$ (size $\binom{9}{3}-12=72$)
        \item $\mathcal A_4=\{1237,1457,2467,3567,1238\}$ (size $\binom{9}{3}-14=70$)
        \item \dots
        \item $\mathcal A_4=\{1237,1457,2467,3567,1238,1458,2468,3568,1239,1459,2469,3569\}$ (size $\binom{9}{3}-28=56$)
    \end{itemize}
    In particular, we obtain flat maximal antichains of all sizes in the interval $[63,79]=\left[\binom{9}{3}-21,\binom{9}{3}-5\right]$ (see \Cref{lem:general_partition} below for the general version). 
\end{example} 
To give a precise statement about the sizes of flat maximal antichains that can be obtained by the above construction we introduce the following notation. For integers $l<t$, let $\phi(t,l)$ be the maximum size of a pairwise shadow-disjoint family $\mathcal F\subseteq\binom{[t]}{l}$. In particular $\phi(t,2)=\lfloor t/2\rfloor$, and in general $\phi(t,l)\geq\left\lceil\frac{1}{t}\binom{t}{l}\right\rceil$ by~\cite[Theorem 1]{Graham1980}. The following lemma is then an immediate consequence of the above construction, and will be sufficient for constructing flat maximal antichains of the required sizes in \Cref{lem:general_partition}. 
\begin{lemma}\label{lem:coding_construction}
  If $n$, $t$, $l$, $s$ and $\alpha$ are non-negative integers with $2\leq l<t<n$, $s\leq\phi(t,l)$ and $s\leq\alpha\leq(n-t)s$, then $\binom{n}{l}-s-\alpha(l-1)\in S(n,l)$.
\end{lemma}
We use \Cref{lem:coding_construction} to find an interval in $S(n,l)$ as follows. For fixed $\abs{\mathcal F}$, varying $\sum_{A\in\mathcal F}\abs{X(A)}$ yields maximal flat antichains whose sizes form an arithmetic progression with common difference $l-1$ (see \Cref{ex:coding_construction} for an illustration). The union of these progressions for varying $\abs{\mathcal F}$ contains an interval whose boundaries are stated in the following lemma. These sizes are what will needed for the final step in the proof of \Cref{thm:main_result} which is provided in \Cref{lem:small_l}.
\begin{lemma}\label{lem:general_partition}
  For $l\geq 2$ and $l+3\leq t\leq n-l$, 
  \[\left[\binom{n}{l}-(1+(l-1)(n-t))\left(\phi(t,l)-l+2\right),\,\binom{n}{l}-(l-1)^2-1\right]\subseteq S(n,l).\]  
\end{lemma}
\begin{proof}
  Let $m=\binom{n}{l}-x$ be any integer in the interval on the left-hand side of the claimed inclusion, that is, $(l-1)^2+1\leq x\leq (1+(l-1)(n-t))(\phi(t,l)-l+2)$. Let $s$ be the largest integer satisfying $s\equiv x\pmod{l-1}$ and $s\leq\min\{\lfloor x/l\rfloor,\,\phi(t,l)\}$, and define $\alpha=(x-s)/(l-1)$. It follows from $t\geq l+3$ and $\phi(t,l)\geq\frac1t\binom{t}{l}$ that $\phi(t,l)\geq l-2$, and therefore, $s\geq 0$. By \Cref{lem:coding_construction} (and noting that $s\leq\alpha$ follows from $s\leq\frac{x}{l}$), it is sufficient to verify $\alpha\leq(n-t)s$, or equivalently,
  \begin{equation}\label{eq:target}
  x\leq \left(1+(l-1)(n-t)\right)s.    
  \end{equation}  
  By construction, $s=\min\{\lfloor x/l\rfloor,\phi(t,l)\}-j$ for some $j\in\{0,1,\dots,l-2\}$. If $s=\phi(t,l)-j$ then (\ref{eq:target}) is immediate from the assumed upper bound on $x$. For $s=\lfloor x/l\rfloor-j$ we will verify $x\leq (1+l(l-1))s$ which implies (\ref{eq:target}) because $n-t\geq l$ by assumption. We write $x=q_1l(l-1)+q_2(l-1)+r$ with integers $q_1\geq 0$, $0\leq q_2\leq l-1$ and $0\leq r\leq l-2$. Then $s$ is the unique integer with $s\equiv r\pmod{l-1}$ and $sl\leq x<(s+l-1)l$, hence
   \[s= 
     \begin{cases}
       q_1(l-1)+r &\text{if }r\leq q_2,\\
       (q_1-1)(l-1)+r & \text{if }r\geq q_2+1.
     \end{cases}
   \]
   If $r\leq q_2$ then
   \begin{multline*}
       (1+l(l-1))s=q_1(l-1)+r+q_1l(l-1)^2+rl(l-1)\\ 
       =q_1l(l-1)+\left[q_1l(l-1)(l-2)+q_1+rl\right](l-1)+r\geq q_1l(l-1)+l(l-1)+r>x,  
   \end{multline*}
   where the term in square brackets is at least $l$ because $(q_1,r)\neq(0,0)$ which follows from $x>(l-1)^2$.    
   If $r\geq q_2+1$, then $q_2\leq l-3$ and $q_1\geq 1$. The value of $(1+l(l-1))s$ is $(1+l(l-1))(l-1)$ less than the one above, that is, 
   \begin{multline*}
       (1+l(l-1))s=q_1l(l-1)+\left[q_1l(l-1)(l-2)+q_1+rl-1-l(l-1)\right](l-1)+r\\
       q_1l(l-1)+rl(l-1)+r\geq q_1l(l-1)+l(l-1)+r>x,
   \end{multline*}
   which concludes the proof.
\end{proof}
We have now all the ingredients to prove \Cref{prop:flat}. For $l=2$ we apply \Cref{lem:general_partition} with
\[t=
  \begin{cases}
    \lfloor n/2\rfloor & \text{if }n\equiv0\text{ or }1\pmod{4},\\
    \lfloor (n+2)/2\rfloor & \text{if }n\equiv2\text{ or }3\pmod{4}.
  \end{cases}  
\]
Then $\phi(t,2)=t/2$ and with $(n-t+1)\frac{t}{2}=\left\lfloor\frac{(n+1)^2}{8}\right\rfloor$, we obtain
\[\left[\binom{n}{2}-\left\lfloor\frac{(n+1)^2}{8}\right\rfloor,\ \binom{n}{2}-2\right]\subseteq S(n,2),\]
and the result follows with \Cref{prop:large_flat}. For $3\leq l\leq (n-2)/2$ we distinguish three cases.
\begin{description}
    \item[Case 1] $l=k-1$. From $\left\lceil\frac{1}{k}\binom{k}{l}\right\rceil-l+2=3-l\leq 0$ it follows that $\gamma=3+\sum_{i=1}^l\frac{1}{i+1}\binom{2i}{i}$, and we are done by \Cref{prop:large_flat}.
    \item[Case 2] $l=k-2$. If $l\geq 5$, then $\left\lceil\frac{1}{k}\binom{k}{l}\right\rceil-l+2\leq 0$, hence $\gamma=3+\sum_{i=1}^l\frac{1}{i+1}\binom{2i}{i}$, and the same can be checked directly for $l\in\{3,4\}$. Hence, \Cref{prop:large_flat} does the job.
    \item[Case 3] $l\leq k-3$. We apply \Cref{lem:general_partition} with $t=k$, and use $\phi(k,l)\geq\left\lceil\frac{1}{k}\binom{k}{l}\right\rceil$, to deduce 
    \[\left[\binom{n}{l}-(1+(l-1)(n-k))\left(\left\lceil\frac{1}{k}\binom{k}{l}\right\rceil-l+2\right),\ \binom{n}{l}-(l-1)^2-1\right]\subseteq S(n,l).\]
    The result follows with \Cref{prop:large_flat} and $(l-1)^2+1\leq 4+\sum_{i=1}^l\frac{1}{i+1}\binom{2i}{i}$.
\end{description}

\subsection{Proof of \texorpdfstring{\Cref{thm:main_result}}{}}\label{subsec:proof_main_thm}
 Let $I_{n,l}$ be the interval which, according to \Cref{prop:flat}, is contained in $S(n,l)$
 ($2\leq l\leq\lfloor\frac{n-2}{2}\rfloor$). Observing that the left end of $I_{n,2}$ is
 $\binom{n}{2}-\left\lfloor\frac{(n+1)^2}{8}\right\rfloor$, and the right end of
 $I_{n,\lfloor(n-2)/2\rfloor}$ is $\binom{n}{k}-k\left\lceil\frac{k+1}{2}\right\rceil$ (recall that $k=\left\lceil\frac{n}{2}\right\rceil$), we would be done if we could show that, for $3\leq l\leq\lfloor\frac{n-2}{2}\rfloor$ the left end of $I_{n,l}$ is at most one more than the right end of $I_{n,l-1}$. If $l\geq 5$ or $n\geq 15$, this is indeed true (as shown in the next lemma). For $l\in\{3,4\}$ and $n\leq 14$, we will then check that the small gaps between the intervals $I_{n,l}$ can be closed using \Cref{prop:large_flat} and \Cref{lem:coding_construction}.   
\begin{lemma}\label{lem:overlap}
  Let $l$ and $n$ be integers with $l\geq 3$ and $n\geq 2l+3$. Assume further that $l\geq 5$ or
  $n\geq 15$. Then (with $k=\left\lceil\frac{n}{2}\right\rceil$)
  \[(n-l)\left\lceil\frac{n-l+1}{2}\right\rceil-1\leq \max\left\{3+\sum_{i=1}^l\frac{1}{i+1}\binom{2i}{i},\,\left(1+(l-1)(n-k)\right)\left(\left\lceil\frac{1}{k}\binom{k}{l}\right\rceil-l+2\right)\right\}.\]
\end{lemma}
\begin{proof}
  For $n\leq 2l+6$, we can assume $l\geq 5$. Bounding the left-hand side from above and the right-hand side from below, it is sufficient to verify
  \[(l+6)\left\lceil\frac{l+7}{2}\right\rceil\leq 3+\sum_{i=1}^{l}\frac{1}{i+1}\binom{2i}{i}.\]
  For $l=5$, this is $66\leq 67$, and for $l\geq 6$ we proceed by induction on $l$:
  \[(l+6)\left\lceil\frac{l+7}{2}\right\rceil\leq(l+5)\left\lceil\frac{l+6}{2}\right\rceil+\frac32l+9\leq 3+\sum_{i=1}^{l-1}\frac{1}{i+1}\binom{2i}{i}+\frac32l+9\leq 3+\sum_{i=1}^{l}\frac{1}{i+1}\binom{2i}{i}.\]
  For $n\geq 2l+7$, we complete the proof by showing
  \[(n-l)\left\lceil\frac{n-l+1}{2}\right\rceil-1\leq\left(1+(l-1)(n-k)\right)\left(\left\lceil\frac{1}{k}\binom{k}{l}\right\rceil-l+2\right).\]
  For $n\leq 26$, we just check all pairs $(l,n)$ by computer. For $n\geq 27$, the following crude bounds are good enough. For the left-hand side,
  \[(n-l)\left\lceil\frac{n-l+1}{2}\right\rceil-1\leq(2k-3)(k-1)-1=2k^2-5k+2,\]
  and for the right-hand side,
\[\left(1+(l-1)(n-k)\right)\left(\left\lceil\frac{1}{k}\binom{k}{l}\right\rceil-l+2\right)\geq(2k-1)\left(\frac1k\binom{k}{3}-k\right)=\frac{2k^3-19k^2+13k-2}{6}.\]
  The claim follows from $2k^3-31k^2+43k-14\geq 0$, and this inequality is a consequence of $2k^3-31k^2+43k-14=(2k-1)(k-1)(k-14)$ and $k\geq 14$. 
\end{proof}
Finally, we deal with the cases $l\in\{3,4\}$ and $n\leq 14$, and prove that the gap between the intervals $I_{n,l-1}\subseteq S(n,l-1)$ and $I_{n,l}\subseteq S(n,l)$ that are obtained from \Cref{prop:flat} is contained in $S(n,l-1)\cup S(n,l)$.
\begin{lemma}\label{lem:small_l}
  Let $l\in\{3,4\}$ and $2l+2\leq n\leq 14$. The integers between the interval in $S(n,l-1)$ and the interval in $S(n,l)$ from \Cref{prop:flat} are contained in $S(n,l-1)\cup S(n,l)$.
\end{lemma}
\begin{proof}
The intervals from \Cref{prop:large_flat}, shown in \Cref{tab:intervals}, are a bit longer, and this is already enough to close the gap in most cases. 
\begin{table}[htb]
\centering
\caption{Intervals from \Cref{prop:large_flat}}\label{tab:intervals}
    \begin{tabular}{cccc}\toprule
         $n$ & interval in $S(n,2)$ & interval in $S(n,3)$ & interval in $S(n,4)$ \\ \midrule
         $8$ & $[18,44]$ & $[45,61]$ & -- \\
         $9$ & $[24,64]$ & $[73,114]$ & -- \\
         $10$ & $[30,96]$ & $[109,190]$ & $[185,240]$ \\
         $11$ & $[37,132]$ & $[132,306]$ & $[305,442]$ \\
         $12$ & $[45,182]$ & $[181,462]$ & $[470,768]$ \\
         $13$ & $[54,236]$ & $[234,677]$ & $[658,1254]$ \\
         $14$ & $[63,309]$ & $[304,951]$ & $[935,1964]$ \\  \bottomrule
    \end{tabular}
\end{table}
The only remaining gaps occur between the interval in $S(n,l-1)$ and the interval in $S(n,l)$ for $(l,n)\in\{(3,9),(3,10),(4,12)\}$, and those are closed using \Cref{lem:general_partition}.
  \begin{itemize}
  \item For $(l,n)=(3,9)$, \Cref{lem:general_partition} with $t=6$ implies $[63,79]\subseteq S(9,3)$.
  \item For $(l,n)=(3,10)$, \Cref{lem:general_partition} with $t=6$ implies $[93,115]\subseteq S(10,3)$.
  \item For $(l,n)=(4,12)$, \Cref{lem:general_partition} with $t=7$ implies $[447,485]\subseteq S(12,4)$.\qedhere
  \end{itemize}
\end{proof}
Combining \Cref{prop:flat} with \Cref{lem:overlap,lem:small_l}, we obtain that
\[\left[\binom{n}{2}-\left\lfloor\frac{(n+1)^2}{8}\right\rfloor,\,\binom{n}{k}-k\left\lfloor\frac{k+1}{2}\right\rfloor\right]\subseteq \bigcup_{l=2}^{\lfloor(n-2)/2\rfloor} S(n,l),\]
and this concludes the proof of \Cref{thm:main_result}.

\section{Open problems}\label{sec:open_problems}
In this paper we proved that almost all sizes of maximal antichains are attained by flat antichains (the exception being sizes below $\binom{n}{2}-\left\lfloor\frac{(n+1)^2}{8}\right\rfloor$). We obtained this by establishing the existence of sufficiently long intervals in the flat maximal antichain size spectra 
\[S(n,l)=\left\{\left\lvert\mathcal A\right\rvert\,:\,\mathcal A\subseteq\binom{[n]}{l}\cup\binom{[n]}{l+1}\text{ is a maximal antichain in }B_n\right\}\]
for $2\leq l\leq \lceil n/2\rceil$.
We think it would be interesting to determine the sets $S(n,l)$ exactly, and construct the corresponding flat maximal antichains. 
\begin{problem}
    Determine the sets $S(n,l)$ and describe the corresponding constructions for flat maximal antichains.
\end{problem}
We already know that $S(n,l)$ contains the long interval from \Cref{prop:flat}, and it remains to be determined what happens above $\binom{n}{l+1}-(n-l-1)\left\lceil\frac{n-l}{2}\right\rceil$ and for $l\geq 3$ also what happens below $\binom{n}{l}-\gamma$. For $l\geq 4$, it follows from \cite[Theorem 2]{Griggs2023} that for $m$ with $\binom{n}{l+1}-(n-l-1)\left\lceil\frac{n-l}{2}\right\rceil<m<\binom{n}{l+1}$, there is a maximal antichain $\mathcal A\subseteq\binom{[n]}{l}\cup\binom{[n]}{l+1}$ of size $\abs{\mathcal A}=m$ if and only if
\begin{equation}\label{eq:gaps}
  m=\binom{n}{l+1}-t(n-l-1)+\binom{a}{2}+\binom{b}{2}+c  
\end{equation}
for some $t\in\{1,2,\dots,n-l+1\}$, $a\geq b\geq c\geq 0$, $1\leq a+b\leq t$. For $l\in\{2,3\}$ this condition on $m$ is still necessary, but its sufficiency does not follow immediately from the results in~\cite{Griggs2023}. 
\begin{problem}
Assume that 
$l\in\{2,3\}$ and $\binom{n}{l+1}-(n-l-1)\left\lceil\frac{n-l}{2}\right\rceil<m<\binom{n}{l+1}$, where $m$ has the form (\ref{eq:gaps}).
    Prove that $m\in S(n,l)$.
\end{problem}
For $l=2$, this would provide a complete description of the set $S(n,2)$: it's the interval from $\min S(n,2)=\binom{n}{2}-\lfloor\frac{(n+1)^2}{8}\rfloor$ to $\binom{n}{3}$ with some gaps above $\binom{n}{3}-(n-3)\lceil\frac{n-2}{2}\rceil$. It seems plausible that $S(n,l)$ has this form also for $l\geq 3$.
\begin{problem}
    For $l\geq 3$, prove that every integer $m$ between $\min S(n,l)$ and $\binom{n}{l+1}-(n-l-1)\lceil\frac{n-l}{2}\rceil$ is an element of $S(n,l)$.
\end{problem}
For a complete description of $S(n,l)$ it would then be sufficient to solve the following problem.
\begin{problem}\label{prob:min_size}
    Determine $\min S(n,l)$, at least asymptotically for $n\to\infty$ and fixed $l\geq 3$.
\end{problem}
This is also interesting with a view towards the recent work on the domination number of the graph induced by two consecutive levels of the $n$-cube \cite{Badakhshian_2019,Balogh2021}, as $\min S(n,l)$ is the independent domination number of this graph. \Cref{prob:min_size} appears to be challenging. For $l\geq 3$, the best bounds we are aware of are from \cite{Gerbner2012}: 
\begin{equation}\label{eq:gerbner_bounds}
  \left(1-\frac{l-1}{l}t_l-o(1)\right)\binom{n}{l}\leq\min S(n,l)\leq\left(1-\frac12\left(\frac{l-1}{l}\right)^{l-1}+o(1)\right)\binom{n}{l},  
\end{equation}
where the \emph{Tur\'an number} $t_l$ is the limit, for $n\to\infty$, of the maximal density of an $l$-uniform hypergraph on $n$ vertices which does not contain a complete $l$-uniform hypergraph on $l+1$ vertices. As a more modest aim, it would be nice to solve the following problem, the second part of which has been stated already in \cite{Gerbner2012}.
\begin{problem}\label{prob:imcremental_improvement}
\begin{enumerate}[(a)]
    \item Find constructions for maximal flat antichains in $\binom{[n]}{l}\cup\binom{[n]}{l+1}$ which are smaller than the upper bound from \cite{Gerbner2012}.
    \item Find lower bounds for $\min S(n,l)$ which do not depend on Tur\'an numbers.
\end{enumerate}    
\end{problem}
The lower bound for $\min S(n,l)$ in (\ref{eq:gerbner_bounds}) is still valid for the domination number without the requirement that the dominating set is independent \cite{Balogh2021}. For $l=3$, the authors of \cite{Balogh2021} provide a construction of a dominating set whose size is asymptotically equal to the lower bound (assuming the conjectured value of $t_3=\frac59$), and they conjecture that the lower bound is asymptotically sharp in general. Hence part (b) of \Cref{prob:imcremental_improvement} is a step towards separating the domination number from the independent domination number for $l\geq 3$.







\printbibliography

\end{document}